\documentclass[12pt,reqno]{amsart}

\usepackage{amsmath}
\usepackage{amsfonts}
\usepackage{amssymb}
\usepackage{amsthm}
\usepackage[top=40pt, bottom=60pt,left=70pt, right=70pt]{geometry}
\usepackage{mathtools}
\usepackage{mathrsfs}
\usepackage{latexsym}
\usepackage{verbatim} 
\numberwithin{equation}{section}
\usepackage{enumitem}

\usepackage{graphicx}
\usepackage{subfigure}
\usepackage{epstopdf}

\usepackage{ifthen} 

\provideboolean{shownotes} 
\setboolean{shownotes}{true} 
\newcommand{\margnote}[1]{
\ifthenelse{\boolean{shownotes}}%
{\marginpar{\raggedright\tiny\texttt{#1}}}%
{}%
}
\newcommand{\hole}[1]{
\ifthenelse{\boolean{shownotes}}%
{\begin{center} \fbox{ \rule {.25cm}{0cm}
\rule[-.1cm]{0cm}{.4cm} \parbox{.85\textwidth}{\begin{center}
\texttt{#1}\end{center}} \rule {.25cm}{0cm}}\end{center}}
{}
}

\theoremstyle{plain}

\newtheorem{lemma}{Lemma}[section]
\newtheorem{theorem}[lemma]{Theorem}
\newtheorem{proposition}[lemma]{Proposition}
\newtheorem{corollary}[lemma]{Corollary}

\theoremstyle{definition}

\newtheorem{remark}[lemma]{Remark}
\newtheorem{definition}[lemma]{Definition}

\theoremstyle{remark}

\usepackage{cite} 

\usepackage[colorlinks=true,urlcolor=blue,
citecolor=red,linkcolor=blue,linktocpage,pdfpagelabels,
bookmarksnumbered,bookmarksopen]{hyperref}

\usepackage{cleveref}

\newcommand{\Id}{\mathrm{I}}

\newcommand{\bm}{\mathbf{m}}
\newcommand{\bH}{\mathbf{H}}
\newcommand{\bh}{\mathbf{h}}

\newcommand{\E}{\mathbb{E}}
\newcommand{\R}{\mathbb{R}}
\newcommand{\C}{\mathbb{C}}

\newcommand{\N}{\mathbb{N}}

\newcommand{\cN}{{\mathcal{N}}}
\newcommand{\cT}{{\mathcal{T}}}
\newcommand{\cE}{{\mathcal{E}}}
\newcommand{\cL}{{\mathcal{L}}}
\newcommand{\cD}{{\mathcal{D}}}
\newcommand{\cF}{{\mathcal{F}}}
\newcommand{\cH}{{\mathcal{H}}}
\newcommand{\cQ}{{\mathcal{Q}}}
\newcommand{\cR}{{\mathcal{R}}}
\newcommand{\cS}{{\mathcal{S}}}
\newcommand{\cK}{{\mathcal{K}}}
\newcommand{\cP}{{\mathcal{P}}}

\newcommand{\cA}{{\mathcal{A}}}
\newcommand{\cJ}{{\mathcal{J}}}

\newcommand{\cG}{\mathcal{G}}

\newcommand{\tcS}{\widetilde{{\mathcal{S}}}}

\renewcommand{\Re}{\mathrm{Re}\,} 
\renewcommand{\Im}{\mathrm{Im}\,}

\newcommand{\sgn}{\mathrm{sgn}\,}
\newcommand{\PV}{\mathrm{P.V.}\,}

\DeclareMathOperator*{\Hess}{\mathrm{Hess}\,}
\DeclareMathOperator*{\rint}{\ThisStyle{\rotatebox{0}{$\SavedStyle\!\int\!$}}}
\usepackage{scalerel}

\newcommand{\ep}{\epsilon}

\newcommand{\brt}{\overline{\theta}}

\newcommand{\ess}{\sigma_\mathrm{\tiny{ess}}}
\newcommand{\ptsp}{\sigma_\mathrm{\tiny{pt}}}

\newcommand{\im}[1]{\textrm{Im}\left(#1\right)}

\newcommand{\pld}[2]{\left \langle #1 \, , #2 \right \rangle_{L^2}}
\newcommand{\phu}[2]{\left \langle #1 \, , #2 \right \rangle_{H^1}}
\newcommand{\pws}[2]{\left \langle #1 \, , #2 \right \rangle_{X}}
\newcommand{\pwse}[2]{\left \langle #1 \, , #2 \right \rangle_{H^1\times L^2}}

\newcommand{\nli}[1]{\left \| #1 \right \|_{L^\infty}}
\newcommand{\nld}[1]{\left \| #1 \right \|_{L^2}}
\newcommand{\nhu}[1]{\left \| #1 \right \|_{H^1}}
\newcommand{\nhd}[1]{\left \| #1 \right \|_{H^2}}
\newcommand{\nwse}[1]{\left \| #1 \right \|_{H^1\times L^2}}
\newcommand{\nws}[1]{\left \| #1 \right \|_{X}}

\newcommand{\res}[1]{\rho\left ( #1 \right )}

\DeclareMathOperator{\lspan}{Span}

\DeclareMathOperator{\opl}{\mathcal{L}}
\DeclareMathOperator{\opli}{\mathcal{L_\infty}}

\DeclareMathOperator{\supp}{supp}

\newcommand{\fint}{\hspace{4pt} \mbox{--}\hspace{-9pt}\int}

\begin{document}
\title[Stability of static N\'{e}el walls]{Nonlinear stability of static N\'{e}el walls in ferromagnetic thin films}

\author[A. Capella]{Antonio Capella}

\address{{\rm (A. Capella)} Instituto de Matem\'aticas\\Universidad Nacional Aut\'onoma de M\'exico\\Circuito Exterior s/n, Ciudad Universitaria\\C.P. 04510 Cd. de M\'{e}xico (Mexico)}

\email{capella@matem.unam.mx}

\author[C. Melcher]{Christof Melcher}

\address{{\rm (C. Melcher)} 
Lehrstuhl f\"ur Angewandte Analysis \\ RWTH Aachen\\D-52056 Aachen (Germany)}

\email{melcher@rwth-aachen.de}

\author[L. Morales]{Lauro Morales}

\address{{\rm (L. Morales)} Instituto de Investigaciones en Matem\'aticas Aplicadas y en Sistemas\\Universidad Nacional Aut\'onoma de M\'exico\\Circuito Escolar s/n, Ciudad Universitaria\\C.P. 04510 Cd. de M\'{e}xico (Mexico)}

\email{lauro.morales@iimas.unam.mx}

\author[R. G. Plaza]{Ram\'on G. Plaza}

\address{{\rm (R. G. Plaza)} Instituto de Investigaciones en Matem\'aticas Aplicadas y en Sistemas\\Universidad Nacional Aut\'onoma de M\'exico\\Circuito Escolar s/n, Ciudad Universitaria\\C.P. 04510 Cd. de M\'{e}xico (Mexico)}

\email{plaza@mym.iimas.unam.mx}

\begin{abstract}
The paper establishes the nonlinear (orbital) stability of static $180$-degree  N\'{e}el walls in ferromagnetic films under the reduced wave-type dynamics for the in-plane magnetization proposed by Capella, Melcher, and Otto \cite{CMO07}. The result follows from the spectral analysis of the linearized operator, which features a challenging non-local operator. As part of the proof,  we show that the linearized non-local linearized operator is a compact perturbation of a suitable non-local linear operator at infinity, a result interesting in itself.  
\end{abstract}

\maketitle
\setcounter{tocdepth}{1}

\tableofcontents


\section{Introduction}
\label{secintro}

In order to study the motion of magnetization vectors in ferromagnetic materials, in 1935 Landau and Lifshitz \cite{LanLif35} introduced a model system of equations, later reformulated and re-derived by Gilbert \cite{Gilb55,Gilb04}, which constitutes the fundamental and best accepted mathematical model that describes the magnetization in ferromagnets. Since ferromagnetic thin films exhibit a wide range of applications to the design and manufacturing of magnetic storage devices, the Landau-Lifshitz-Gilbert (LLG) model has attracted a great deal of attention from physicists and mathematicians alike in the last decades. A great variety of patterns of magnetization vectors appear in ferromagnetic films. For instance, narrow transition regions between opposite magnetization domains are called \emph{domain walls}. Some of the most common wall types in such materials are called N\'eel walls, separating two opposite magnetization regions by an in-plane rotation, oriented along an axis; Bloch walls, for which the magnetization rotates about the normal of the domain wall, pointing along the domain wall plane in a 3D system; or Walker walls, which are formed under the presence of an external magnetic field (see, e.g., Hubert and Sch\"{a}fer\cite{HuSch98} for further information).

One of the main objectives of recent mathematical studies is to understand the behavior of these dynamical coherent structures developed by the magnetization of a ferromagnet. The stability under small perturbations of these microstructures is important, not only to validate the mathematical model but also to enhance the numerical simulations performed by physicists and engineers to optimize and design new ferromagnetic materials (see, e.g., \cite{LabBer99}). Up to our knowledge, the literature on the dynamical stability theory for magnetic domain walls is scarce. The stability of one-dimensional Bloch walls has been addressed by Krukowski \cite{Kruk87} using a spectral (linearized) calculation of energies of ground states, and by Carbou and Labb\'e \cite{CarLab06}, under the nanowire, one-dimensional approximation by Sanchez \cite{Sanch09}. Takasao \cite{Tak11} improved this last result for Walker walls, also in one dimension and in the presence of an external magnetic field. Carbou \cite{Carb10} proved the stability of a Walker wall in the three-dimensional model using the energy method and under a simplifying assumption that gets rid of the non-local part of the operator. Most of these works employ energy methods to conclude stability, that is, the analyses are based on performing \emph{a priori} energy estimates on the equations of magnetization evolution and relying on their intrinsic structure. 

This paper is devoted to studying the dynamical stability of static N\'eel walls. Our departure point is the one-dimensional thin film reduction of the micromagnetic energy proposed by Capella, Melcher, and Otto \cite{CMO07} (outlined previously in \cite{MuOs06} for numerical purposes), which establishes an effective system for the in-plane magnetization by taking the thin film layer limit. The resulting system underlies a wave-type dynamics for the N\'eel wall's phase. The authors prove the existence and uniqueness of a static N\'eel wall's phase profile in the absence of external fields, as well as the emergence of traveling wave solutions near the static profile under the influence of a small constant external forcing. The authors also outline the stability of these structures under small one-dimensional perturbations. The present analysis constitutes a follow-up of such formulation and a full study of the nonlinear stability of the static N\'eel wall under small, one-dimensional perturbations of the phase itself. As far as we know, this problem has not been studied before in the literature. 

One of the main technical difficulties pertains to the non-locality of the dynamical equation, even at a linear level. In contrast to previous studies, we adopt a spectral approach to the problem. Motivated by the ideas in \cite{CMO07}, in which the linearized operator around the static phase is defined and previously studied, we profit from this information and perform a full spectral stability analysis of this operator, that includes a proof of its relative compactness with respect to an asymptotic operator. In contrast with standard techniques, which are usually applied to local differential operators with bounded coefficients and which are based on truncating such coefficients with their asymptotic limits (see, e.g., \cite{KaPro13}, Section 3.1), in this work and by necessity (because we are studying a non-local operator) we develop a novel procedure that focuses on describing totally bounded sets in terms of $L^2$-equicontinuity and uniform decay in Fourier space (see Theorem \ref{main_res_L-Linf} below). This relative compactness plays a crucial role in the location of the essential spectrum of a block matrix operator matrix that encodes the linearization of the nonlinear wave equation for perturbations of the static wall. It is proved that both essential and point spectra are stable, that is, they belong to the stable half-plane of complex numbers with negative real part, except for the origin, which is associated with translations of the N\'eel wall (see Theorem \ref{mainspthm}). An important feature is the presence of an \emph{spectral gap}, which is a positive distance from the eigenvalue zero to the rest of the spectrum. This allows us to establish the exponential decay of the solutions to the spectral problem when projected outside the one-dimensional vector space generated by translations of the static profile. Upon application of the well-known Gearhart-Pr\"uss theorem \cite{CrL03,EN00} and after the establishment of uniform resolvent estimates, we conclude that the semigroup generated by the linear block matrix operator is exponentially decaying in the appropriate subspace. This information is then used to prove nonlinear stability. For that purpose, we apply an abstract result, due originally to Sattinger \cite{Sa76} and adapted to a Hilbert space setting by Lattanzio \emph{et al.} \cite{LMPS16}, that establishes nonlinear stability from spectral stability by controlling the growth of nonlinear terms and profiting from the fact that the manifold generated by the wave is one-dimensional (the group of translations). We regard our contributions not only new in the context of ferromagnetic wall stability analysis, but also of methodological nature: we advocate for spectral and nonlinear analysis as a feasible and effective method in the study of this type of problems. The unpublished note by Huber \cite{Hub10} warrants note as the only work (as far as we know) that performs a rigorous spectral analysis of the linearized operator around a N\'eel wall for a layer of small (but positive) thickness, $\ep > 0$. Huber does not prove the spectral stability of this structure but employs the spectral information to obtain time-periodic solutions in a vicinity of it. (We note that in layers with positive thickness, the linearized operators are sectorial, in contrast with the present case of a thin-film limit.)

\subsection*{Plan of the paper}

This paper is structured as follows. Section \ref{secprel} contains a brief description of the thin-film dynamical model in \cite{CMO07}, recalls some of the main properties of the static N\'eel wall's phase, and states the main result of this paper. Section \ref{sec:strategy} presents a general overview of the proof's strategy. Section \ref{seclinearopL} is devoted to study of the linearized (scalar) operator around the static N\'eel wall defined in \cite{CMO07}. In particular, it is shown that it is relatively compact to an asymptotic operator, a feature that plays a key role in the stability analysis. Section \ref{secspectral} establishes the spectral stability of the N\'eel wall's phase. The spectral problem is posed in terms of a block operator and the stability of both its essential and point spectra is established. Section \ref{secsemigroup} is devoted to proving the existence of an associated semigroup to the dynamic problem and showing the exponential decay of solutions to the linearized equations outside a one-dimensional space related to profile translations. The final Section \ref{sec:nonlinear_stability} contains the proof of Theorem \ref{maintheorem}.

\subsection*{Notations}
Along this manuscript, we denote the spaces  $L^2(\R, \C), \ H^1(\R, \C)$ and $H^2(\R, \C)$ of complex-valued functions by $L^2, \ H^1$ and $H^2$. Meanwhile, their real-valued version are denoted by $L^2(\R), \ H^1(\R)$ and $H^2(\R)$ respectively. The set of unitary vectors in $\R^n$ is denoted by $\mathbb{S}^{n-1}$.  The operators $\hat{\cdot}:L^2\to L^2$ and $\check{\cdot}:L^2\to L^2$ stand for the Fourier transform and its inverse, respectively. Also, $\xi$ represents the variable in the frequency domain. In the same fashion, the half-Laplacian is defined by the relation $(-\Delta)^{1/2}u = (|\xi|\hat{u})\check{}$, and $\|u\|_{\dot{H}^{1/2}}$ denotes the fractional $H^{1/2}$-norm of the function $u\in L^2$ given by $\|u\|_{\dot{H}^{1/2}}:=\nld{|\xi|^{1/2}\hat{u}}$. Finally, if $X$ is a Banach space and $\cA$, $\cT$ are two operators in $X$ such that $D(\cA) = D(\cT)$  then, the commutator $[\cA, \cT]:D(\cA)\subset X\to X$ is given by the difference $\cA \cT- \cT\cA $. Moreover, if $K\subset D(\cA)$, then the direct image set of $K$ under the operator $\cA$ is denoted by $\cA\, K$, namely $\cA\, K:=\{ \cA u \, |\, u\in K\}$. 

\section{Preliminaries and main result}
\label{secprel}

\subsection{The micromagnetic model}
The Landau and Lifshitz continuum theory of ferromagnetic materials\cite{LanLif35} is based on a magnetization field ${\bf m}: \tilde\Omega\to {\mathbb S}^2$, that represents the local average magnetic moment,  and a variational principle in terms of the {\it micromagnetic energy}. In the absence of an external field, the micromagnetic energy is given by 
\[
 \E(\bm) = \frac{1}{2}\Big( d^2 \int_{\widetilde{\Omega}} |\nabla \bm|^2 \, dx  + Q \int_{\widetilde{\Omega}} \Phi(\bm) \, dx + \int_{\R^3} |\nabla U|^2\Big).
\]
where $d>0$ and $Q>0$ are constants. The first  term is known as the {\em exchange energy} and it quantifies the energy due to magnetic-dipole interactions in the sample. So, $d$ is called as the {\em exchange length}. The second  term penalizes crystalline anisotropy via $\Phi$ which usually has the form of an even polynomial in $\bm \in \mathbb{S}^2$. Then, $Q$ measures the relative strength of anisotropy penalization. 
Finally, the last  term in the expression of $\E(\bm)$ is called the {\em stray energy}. The term $\nabla U$ is called the \textit{stray field} and it is uniquely defined via the distribution equation $\Delta U = \textrm{div}\,(\bm \boldsymbol{1}_{\tilde \Omega})$ ($\boldsymbol{1}_A$ denotes the indicator function of the set $A$).
The stray-field energy favors vanishing distributional divergence, namely, $\textrm{div} {\bf m} = 0$ in $\tilde\Omega$ and ${\bf m}\cdot n = 0$ on $\partial\tilde\Omega$, where $n$ is the outward normal to the boundary. 
The combination of the stray-field energy (which is a non-local term) and the non-convex saturation constraint $|{\bf m}|=1$ gives rise to pattern formation among magnetic domains where the magnetization is almost constant. Thin transition layers separating the magnetic domains are known as domain walls and may form complex patterns \cite{HuSch98}. 

\subsection{Stationary N\'{e}el wall profile in soft magnetic thin films}
A thin film is an infinitely extended magnetic material $\tilde\Omega = \R^2\times (0,\delta)$ where $\delta \ll d$. In this regime,  it is safe to assume that the magnetization is independent of the $x_3$ variable. We assume further that the magnetization is $\ell$-periodic in the ${\bf e}_2$ direction, namely,
$$
{\bf m}(x_1,x_2+\ell) = \bm(x_1,x_2) \quad\text{for any } x = (x_1,x_2)\in\R^2, 
$$ 
that the material has a uniaxial anisotropy in the $e_2$ direction, with $\Phi({\bf m}) = 1- m_2^2$. We consider transition layers connecting antipodal states on the easy axis
$$
{\bf m}:\R^2\to \mathbb{S}^2 \quad\text{ with } \bm(\pm \infty, x_2) = (0,\pm 1, 0) \quad
\text{for any } x_2\in \R.
$$
In this case, the stray energy is approximated at leading order by 
$$
E_{stray}({\bf m})= \frac{1}{2} 
\fint_0^\ell\int_\R \left(\frac{\delta}{2} \left| 
|\nabla|^{\frac{1}{2}} {\mathcal H}(m)
\right|^2  + m_3^2 \right) dx,
$$
where ${\bf m} = (m, m_3)$ with $m=(m_1,m_2)$ and formally ${\mathcal H}(m)= \nabla \Delta^{-1}\text{div } m$ see \cite{CMO07}. Thus, the micromagnetic energy becomes 
$$
\E({\bf m})= \frac{1}{2} 
\fint_0^\ell\int_\R \left( d^2|\nabla m |^2 +   \frac{\delta}{2} \left| 
|\nabla|^{\frac{1}{2}} {\mathcal H}(m)
\right|^2 + Q(1-m_2)^2 + m_3^2  \right) dx.
$$

N\'eel walls are one-dimensional transition layers observed in soft ferromagnetic thin films, that is, magnetic materials with relatively weak anisotropic energy.   Here, we consider a parameter regime of soft thin films so that the anisotropy and relative thickness  are balanced, more precisely 
\begin{equation}\label{reg-Neel}
    Q \ll 1, \quad \kappa^{-1} = \delta/d \ll 1  \quad \text{while } {\mathcal Q} = 4 \kappa^2 Q.
\end{equation}
Therefore, it is feasible to introduce the small parameter $\varepsilon = \sqrt{Q}$. By rescaling the length $x$ by $w = \delta /(2Q)$, and the energy by $\delta/2$, the micromagnetic energy becomes
\begin{equation}\label{Energy-epsilon}
E_\varepsilon ({\bf m}) = \frac{1}{2} 
\fint_0^L\int_\R \left({\mathcal Q}|\nabla m |^2 +    \left| 
|\nabla|^{\frac{1}{2}} {\mathcal H}(m)
\right|^2 + (1-m_2)^2 + \left(\frac{m_3}{\varepsilon}\right)^2  \right) dx,
\end{equation}
where $L= \ell/w$ and we assumed $\varepsilon \ll {\mathcal Q} \lesssim 1$.
Assuming further that $m=m(x_1)$ then ${\mathcal H}(m) = m_1 {\bf e}_1$ is independent of $x_2$ and the reduced variational principle for the one-dimensional wall transition is
 \begin{equation}\label{varppio}
 \begin{aligned}
E_\varepsilon ({\bf m}) = \frac{1}{2} 
\int_\R  & \Big(
{\mathcal Q} | {\bf m}'|^2 
+| |\nabla|^{\frac{1}{2}} m_1|^2 
+ (1-m_2)^2 
+ \left(\frac{m_3}{\varepsilon}\right)^2  \Big) dx \to \min,\\ 
 &{\bf m} : \R \to {\mathbb S}^2 \qquad \text{with} \;\;  {\bf m}(\pm\infty) =(0,\pm 1,0),
 \end{aligned}
\end{equation}
where ${\bf m}'=\frac{d{\bf m}}{dx_1}$.
In \cite{GC04} it is shown that for $\varepsilon_k\to 0$ there exists a sequence of minimizers ${\bf m}_{\varepsilon_k}$ of 
\eqref{varppio} with a locally convergent subsequence to ${\bf m} = (m,0)$ and satisfies 
 \begin{equation}\label{limit-varppio}
 \begin{aligned}
E_0 (m) &= \tfrac{1}{2} \left({\mathcal Q} 
\Vert m'\Vert_{L^2(\R)}^2 + \Vert m_1 \Vert_{\dot{H}^{1/2}(\R)}^2 + \Vert m_1\Vert_{L^2(\R)}^2    \right) \to \min,\\ 
 &m : \R \to {\mathbb S}^2 \qquad \text{with} \;\;  
 m(\pm\infty) =(0,\pm 1).
 \end{aligned}
\end{equation}
Since the left translation is an $L^2$-isometry, the expression of $E_0(m)$ is invariant spatial translations. This invariance is inherited by the energy, yielding that minimizers of \eqref{limit-varppio} are unique up to  translations. Despite this invariance, $E_0(m)$ is a strictly convex functional on $m_1$  because  $|m'|^2=(m_1')^2/(1-m_1^2)$. Thus, the variational principle \eqref{limit-varppio} has a minimizer for any ${\mathcal Q}>0$. The minimizer that satisfies $m_1(0)=1$ is called the {\em N\'eel wall profile}. We refer to $E_0(m)$ as the  {\em N\'eel wall energy}.

For our analysis, 
we introduce the phase $\theta:\R \to \R$ so that $m=(\cos\theta, \sin\theta)$ and the variational principle \eqref{limit-varppio}  becomes  
\begin{equation}
 \label{varprob}
 \begin{aligned}
 \mathcal{E}(\theta) &= \frac{1}{2} \big( {\mathcal Q}\|\theta'\|_{L^2}^2 + \|\cos \theta\|^2_{{\dot H}^{1/2}} + \|\cos \theta\|_{L^2}^2 \big) \; \rightarrow \;  \min\\ 
 \theta : \R &\to (-\pi/2,\pi/2), \qquad \text{with} \;\; \theta(\pm\infty) = \pm \pi/2.
 \end{aligned}
\end{equation}
Since we are interested in N\'eel wall's dynamic,  we refer to minimizers of \eqref{varprob} as the \emph{static N\'eel wall's phase}. From now on, we assume ${\mathcal Q}=1$ . Despite $\theta$ is a function of one variable, we abuse notation by letting $\partial_x\theta=\theta'$ and $\partial_{x}^2\theta =\theta''$.

The following proposition summarizes the basic properties of the static N\'eel wall phase.    

\begin{proposition}[properties of the static N\'eel wall's phase \cite{CMO07,Melc03}]
\label{propNeelw}
 There exists a static N\'eel wall solution with phase $\brt = \brt(x)$, $\brt : \R \to (-\pi/2,\pi/2)$, satisfying the following:
 \begin{itemize}
  \item[(a)] 
  \label{propa} $\brt$ is a strict minimizer of the variational problem \eqref{varprob}, with center at the origin, $\brt(0) = 0$, and monotone increasing, $\partial_x \brt > 0$ $\, \forall x \in \R$.
  \item[(b)]\label{propb} $\brt$ is a smooth solution to
  \begin{equation}\label{ELeq}
   \partial_x^2 \theta + \sin \theta (1+(-\Delta)^{1/2}) \cos \theta = 0,
  \end{equation}
which is the Euler-Lagrange equation for the variational problem \eqref{varprob}.
\item[(c)]\label{propc} $\partial_x \brt \in H^2$.
\item[(d)]\label{propd} For every $u \in H^1$ such that $u(0) = 0$ there holds
\begin{equation}
 \label{Hesspos}
 \Hess \mathcal{E}(\brt) \langle u,u \rangle_{L^2} \geq \|u \, \partial_x \brt\|^2_{L^2} + \Re \, b [u\sin \brt, u \sin \brt],
\end{equation}
where the bilinear form $b[\cdot,\cdot] : H^1 \times H^1 \to \C$, defined as,
\begin{equation}
\label{defbilinearB}
 b[f,g] = \int_\R (1+|\xi|) \hat f(\xi) \hat g(\xi)^* \, d\xi, \qquad f, g \in H^1,
\end{equation}
is equivalent to the standard inner product in $H^{1/2}$.
\end{itemize}
\end{proposition}
\begin{proof}
Property \hyperref[propa]{\rm{(a)}} results from combining Lemma 1 in \cite{CMO07} with the main results of \cite{Melc03} (Propositions 1 and 2). The proof of the smoothness of the N\'eel wall can be found in \cite{Melc04} (Proposition 2). Since $u$ is a minimizer, it satisfies equation \eqref{ELeq} (see Lemma 1 in \cite{CMO07}). This shows \hyperref[propb]{\rm{(b)}}. Moreover, it is proved in \cite{CMO07} (Theorem 1 and Lemma 1) that $\partial_x \brt$, $\partial_x^2 \brt \in L^2(\R)$. As pointed out by the authors, from the Euler-Lagrange equation \eqref{ELeq} the regularity arguments of Lemma 1 can be bootstrapped to show that $\partial_x^3 \brt \in L^2(\R)$. This shows \hyperref[propc]{\rm{(c)}}. Finally, property \hyperref[propd]{\rm{(d)}} is the content of Lemma 3 in \cite{CMO07}.
\end{proof}

\begin{corollary}
\label{corthetabded}
The N\'eel wall's phase $\brt$ belongs to $W^{2,\infty}(\R)$ 
\end{corollary}
\begin{proof}
The result follows immediately from the facts that $|\brt|\leq \pi/2$,  $\partial_x \brt \in H^2$ (see Proposition \ref{propNeelw} \hyperref[propc]{\rm{(c)}}) and the Sobolev's inequality: $\| u \|_{L^\infty}^2 \leq 2 \|u \|_{L^2} \| \partial_xu \|_{L^2}$ for all $u \in H^1$.
\end{proof}

\subsection{LLG dynamics}
The time evolution of the magnetization distribution on a ferromagnetic body $\widetilde{\Omega} \subset \R^3$ is governed by the Landau-Lifshitz-Gilbert (LLG) equation \cite{LanLif35,Gilb55,Gilb04}:
\begin{equation}
\label{eqLLG}
\bm_t+ \alpha \bm \times  \bm_t-\gamma \bm \times \bH_{\mathrm{eff}} = 0, 
\end{equation}
where $\bm : \widetilde{\Omega} \times (0,\infty) \to \mathbb{S}^2 \subset \R^3$ is the magnetization field, $\alpha > 0$ is a non-dimensional damping coefficient (Gilbert factor), and $\gamma > 0$ is the (constant) absolute value of the gyromagnetic ratio with dimensions of frequency (see, e.g., \cite{Gilb04}). The effective field, $\bH_{\mathrm{eff}} = \bh - \nabla \E(\bm)$, is the applied field $\bh$ and the negative functional gradient of the micromagnetic energy $\E(\bm)$. If we consider a single magnetic spin $\bm=\bm(t)$ under a constant magnetic field $\bh$ and neglect damping then, the magnetization $m$ will precess about the applied field $\bh$ with a frequency given by $\omega = \gamma |\bh|$.  When the damping is turned on, the vector $\bm$ will spiral down around $\bh$ until $\bm$ and $\bh$ become parallel. The typical relaxation time is $1/(\alpha\omega)$. 

In bulk materials and up to translations, there exists a one-dimensional optimal path connecting antipodal magnetization states known as the Bloch wall. Bloch walls are such that $m_1=0$ and the transition is perpendicular to the transition axis. In this case, the magnetization $\bm$ is divergence-free and the stray field energy vanishes. Now, for an applied external magnetic field $\bh = H {\bf e}_2$  the landscape changes because under the same initial conditions, explicit dynamic solutions are found and they show that the magnetization $\bm$ rotates to develop a non-vanishing $m_1$ component. The latter component implies a rotation of the other magnetization components advancing the domain wall \cite{HuSch98, Melc04}.  

\subsection{LLG wave-type dynamic thin film limit}
Thin films are incompatible with gyrotropic wall motion due to the incompatibility constraint of the in-plane magnetization imposed by the stray field. In this configuration, the competition between energy and dynamic forces becomes singular in the thin field limit. In \cite{CMO07}, an effective suitable limit is considered under the appropriate regime where the oscillatory features of the LLG dynamics are preserved in the limit. It turns out that the effective dynamics depend on the asymptotic regime as $\alpha$ and the relative thickness $\delta/d$ tend to zero. 

For the precise scaling and regime in \cite{CMO07} let $\varepsilon = \sqrt{Q}$ and consider \eqref{reg-Neel} when $\varepsilon \ll {\mathcal Q}$ while ${\mathcal Q} = (2\varepsilon d/\delta)^2 \lesssim 1$
is small but bounded from below. That is, $\varepsilon\sim \delta/d$ can be regarded as the relative thickness. Under these assumptions, we rescale space, time, and energy by 
$$
x\mapsto w x,\quad  t \mapsto t/(\gamma\varepsilon), \quad\text{and }\quad E_\varepsilon = (2/\delta) \E(\bm),
$$
where $w = \delta/(2\varepsilon^2)$. In this scaling, the mean effective field $\bH_{\mathrm{eff}}$ becomes
\[ 
    \bH_{\mathrm{eff}} = -\varepsilon^2\nabla E_\varepsilon(\bm).
\]
Notice that $E_\varepsilon(\bm)$ is given by \eqref{Energy-epsilon}. Therefore, the LLG equation \eqref{eqLLG} becomes, 
\begin{equation}\label{LLG-epsilon}
    \bm_t + \alpha \bm\times \bm_t + \varepsilon \bm \times \nabla E_\varepsilon(\bm) = 0.
\end{equation}
 To derive the effective equation for the in-plane magnetization it is necessary to write down  $E_\varepsilon(\bm)$ in terms of $m=(m_1,m_2)$ and  $m_3$, that is, 
\begin{equation*}
    E_\varepsilon(\bm) = E_0(m) 
    + \frac{1}{2}\fint_0^L\int_\R\Big({\mathcal Q} |\nabla m_3|^2 + \left(\frac{m_3}{\varepsilon}\right)^2\Big) dx 
\end{equation*}
where 
\begin{equation} \label{E_02D}
E_0(m) = \frac{1}{2}\fint_0^L\int_\R\Big(
{\mathcal Q}|\nabla m|^2 + ||\nabla|^\frac{1}{2}{\mathcal H}(m)|^2 + (1-m_2^2)
\Big) dx.
\end{equation}
Notice that for one-dimensional transition layers the energy $E_0$ coincides with the reduced N\'eel wall energy \eqref{limit-varppio}. 

In \cite{CMO07} it is shown that as 
\[
    \varepsilon\to 0 \quad\text{while}\quad 
    \alpha(\varepsilon)/\varepsilon \to \nu
\]
for some positive $\nu$, while keeping ${\mathcal Q}=1$ for every $\varepsilon >0$, there exist a sequence of solution $\bm_\varepsilon$ of \eqref{LLG-epsilon} $L$-periodic in the $x_2$ direction such that the in-plane magnetization $m_\varepsilon$ weakly converges to $m\in {\mathbb S}^1 $, (in the appropriate spaces) a weak solution of 
\begin{equation}\label{LLG-limit}
[\partial_t^2 m + \nu \partial_t m + \nabla E_0(m)]\perp {T_m \mathbb S}^1.  
\end{equation}

Because $E_0(m)$ coincides with the N\'eel wall energy, it is clear that under the appropriate boundary conditions at infinity (e.g. \eqref{limit-varppio}) the static N\'eel wall profile $\bar m=(\cos\bar\theta,\sin\bar\theta)$ is a static solution of \eqref{LLG-limit}. 

\subsection{Main result}

The static N\'eel wall solution and the wave-type dynamic equation \eqref{LLG-limit} are the starting point of the present work.  
We state our main result in terms of the magnetic phase $\theta:\R\times (0,\infty) \to \R$. As function of $\theta(x,t)$, equation \eqref{LLG-limit} with the boundary conditions given by \eqref{varprob} becomes 
\begin{equation}
 \label{reddyneq}
 \left\{ \ \ 
\begin{aligned}
&\partial_t^2 \theta + \nu \partial_t \theta + \nabla \cE(\theta) = 0, \\
&\theta(-\infty,t) =-\pi/2,\quad \theta(\infty,t) =\pi/2,\\
&\theta(x,0) =\theta_0(x) ,\quad \partial_t\theta(x,0) =v_0(x),
\end{aligned}
\right.
\end{equation}
where $(\theta_0,v_0)$ are some initial conditions and the energy $\cE(\theta)$ is as in \eqref{varprob}. After these definitions, we are ready to state our main result.

\begin{theorem}[Orbital stability of the static N\'eel wall]
\label{maintheorem}

Let $\cJ\subset H^1(\R)\times L^2(\R)$ be the set of initial conditions such that the Cauchy problem  \eqref{reddyneq} has a global solution. There exists $\varepsilon > 0$ sufficiently small such that if the pair $(\theta_0, v_0) \in \cJ$ satisfies
\[
\| \theta_0 - \brt \|_{H^1} + \| v_0 \|_{L^2} < \varepsilon,
\]
then, the solution to \eqref{reddyneq} with initial condition $(\theta(x,0), \partial_t \theta(x,0)) = (\theta_0, v_0)$ satisfies
\[
\| \theta (\cdot, t) - \brt(\cdot + \delta) \|_{H^1} \leq C \exp(- \omega t),
\]
for any $t > 0$, where $\delta \in \R$ is some shift and $C, \omega > 0$ are some constants that may depend on $(\theta_0, v_0)$ and $\varepsilon$.
\end{theorem}

\begin{remark}
It is to be noticed that we are not proving the global existence of the solution for a given small initial perturbation. Theorem \ref{maintheorem} states that any eventual initial small perturbation of the static N\'eel profile, if exists, must decay to a translation of it. This type of behavior is also called  \emph{orbital} stability (or stability \emph{in shape}), as initial perturbations decay to an element of the orbit or manifold generated by the static wave which, in this case, is the one-dimensional manifold of translations. The existence of global solutions can be studied using standard semigroup techniques and with the help of the decaying estimates performed in this work; we do not pursue such analysis here. Instead, we focus on the stability problem alone. 
\end{remark}

\begin{remark}
The perturbations considered in Theorem \ref{maintheorem} only affect the in-plane magnetization $m=(m_1,m_2)$ through the phase $\theta$. In the regime where equation \eqref{reddyneq} is derived, perturbations on the out-of-plane component $m_3$ have infinite energy ( see equation \eqref{Energy-epsilon} and recall that $\epsilon\to 0$)  and are not part of the model.
\end{remark}

\section{Strategy of the proof}
\label{sec:strategy}
This work aims to establish the orbital stability of the Neel wall profile, i.e., Theorem \ref{maintheorem}. We address this problem by rephrasing 
equation \eqref{reddyneq} as a two-dimensional ODE system
\begin{equation*}
\partial_t W = F(W), 
\end{equation*}
where $\varphi = \partial_t \theta$,  $W = (\theta,\varphi)$, $W_0 = (u_0,v_0)$,  and $F(W) = (\varphi,-\nu \varphi - \nabla \cE(\theta))$.
In this setting, we apply the result of  Lattanzio {\em et al.} \cite{LMPS16} (Theorem~\ref{LPstability}) to prove orbital stability. The conditions of Lattanzio's theorem are the following: 

\begin{enumerate}[label=(\alph*)]
\item The flow $F$ vanishes on a one-dimensional manifold $\phi$ that contains the origin.

\item The difference between $\phi$ and its linearization at the origin has a supralinear growth in a neighborhood of the latter. 

\item The difference between $F$ and its linearization at $\phi$ has a supralinear growth in a neighborhood of the one-dimensional manifold. 

\item The linearization of the nonlinear ODE system around any point on the one-dimensional manifold $\phi$ has a solution whose projection on a suitable subspace of codimension one decay exponentially. 
\end{enumerate}

Neel walls's translation define the one-dimensional manifold in condition (a), namely $\phi(\delta) = (\brt(\cdot+\delta),0)$.  Because the Neel wall energy is invariant under translation, the zero-level set of $F$ contains the manifold $\phi$. Condition (b) holds because the  N\'eel wall's phase inherited its smoothness to the one-dimensional manifold $\phi$. 
In addition, $\nabla \cE$ is regular enough to grant that condition (c) holds. The complete proof of (a), (b), and (c) are presented in Section \ref{sec:nonlinear_stability}.

The proof that condition (d) holds is the most challenging; most of the paper is devoted to it (sections \ref{seclinearopL} to \ref{secsemigroup}). 
In the proof of condition (d), we are only concerned with the linearized ODE system, 
\begin{equation*}
\partial_t U = \cA^\delta U, \qquad \mbox{where} \qquad \cA^\delta: = \begin{pmatrix} 0 & \Id \\ -\opl^\delta & -\nu\Id \end{pmatrix}, \qquad
\end{equation*}
and  $\opl^\delta$ is the linearization of $\nabla \cE$ around $\brt(\cdot+\delta)$. Because we rely on the semigroup theory for the existence of solutions, we will concentrate on the operator $\cL:=\cL^0$ for the untranslated N\'eel wall phase $\brt$ and its associated block operator $\cA:=\cA^0$, since the whole operator family $\{\cL^\delta \ | \ \delta \in \R \}$ is isospectral \cite{KaPro13}.

Although the semigroup theory holds under rather general conditions, the nonlocal character of $\cA$ poses challenging problems.   
The proof of condition (d) follows by showing, first, that block matrix operator $\cA$ is the generator of a $C_0$-semigroup, and second, that the dynamics perpendicular to the central manifold is exponentially decaying. We have at least two strategies for the former step: apply the generalized \emph{Hille-Yosida theorem}, see \cite{EN00}, or the classical \emph{Lumer-Phillips theorem}, see e.g., \cite{ReRo04}. 
We take the latter approach(Section \ref{secsemigroup}) because the Hille-Yosida theorem requires bounds on the norm for all powers of the resolvent, but in our case, these bounds are hard to prove due to the nonlocal term.

For a linear and closed operator, $\cA:X \to Y$, between two Banach spaces, the resolvent set is defined as 
\[
\rho(\cA) := \{\lambda \in \C \, : \, \cA - \lambda \,\text{ is injective and onto, and } (\cA - \lambda)^{-1} \, \text{is bounded} \, \},
\]
and the spectrum $\sigma(\cA) = \C \backslash \rho(\cA)$ is the complex complement of the resolvent \cite{Kat80}. By Weyl's spectrum splitting~\cite{We10}, $\sigma(\cA) = \ess(\cA) \cup \ptsp(\cA)$ where 
\begin{equation*}
	\begin{aligned}
		\ptsp(\cA) &:= \{ \lambda \in \C\,: \; \cA - \lambda \,\text{ is Fredholm with index zero and non-trivial kernel} \},\\
		\ess(\cA) &:= \{ \lambda \in \C\,: \; \cA - \lambda \,\text{ is either not Fredholm or
			has index different from zero} \}.
\end{aligned}
\end{equation*}
In this splitting the point spectrum $\ptsp(\cA)$  is a discrete set of {\em isolated eigenvalues} \cite{Kat80,KaPro13}.

Because the spectral mapping theorem is not valid for $C_0$-semigroups, see \cite{EN00}, we use {\em Gearhart-Pr\"uss theorem}  to prove the second step in the proof of condition (d). We recall that Gearhart-Pr\"uss theorem states that any $C_0$-semigroup $\{e^{t\cT}\}_{t\geq 0}$ on a Hilbert space $H$ is exponentially decaying if and only if $\cT$'s spectral bound is negative and its resolvent satisfies $\sup_{\Re \lambda > 0} \|(\cT - \lambda)^{-1}\| < \infty$. We present the proof of these condition in Lemma \ref{elcoro} and Lemma \ref{lemF6} for a suitable restriction of the operator $\cA$, respectively.
In our problem, $\sigma(\cA)$ and $\sigma(\cL)$ are linked.  Furthermore,  if $\lambda \in\ptsp(\cA)$, then $-\lambda(\lambda +\nu)\in \sigma(\cL)$, (Subsection \ref{subsec:pointstability}). Therefore, we only need to characterize the spectrum of $\cL$ to localize $\ptsp(\cA)$. Although $\cL$ is a nonlocal operator, its spectrum is easily localized since it is self-adjoint and coercive on the $L^2$-orthogonal complement of $\partial_x \brt$ (Theorem \ref{main_res_L}).  
Despite $\ess(\cA)$ also depends on $\sigma(\cL)$, its localization is not straightforward. We rely on Weyl’s essential spectrum theorem, \cite{KaPro13},  that implies $\ess(\cA)=\ess(\cA_\infty)$ provided $\cA$ is a relatively compact perturbation of an auxiliary operator $\cA_\infty$, (Definition \ref{initial_def} \ref{def_relative_perturbation}).
As $\cA$, the auxiliary operator $\cA_\infty$ has a block structure written in terms of a positive-defined and self-adjoint linear operator $\cL_\infty$. In addition, since $\cL_\infty$ is an invertible operator, the localization of $\cA_\infty$'s complete spectrum is possible. As expected,  $\sigma(A_\infty)$ is contained in the negative real part complex semi-plane. 

Ultimately, due to the structure of the block matrix operators, the relative compactness between $\cA$ and $\cA_\infty$ follows from the operator's $(\cL_\infty-\cL): H^2\to L^2$ compactness. Therefore, we present the spectral analysis of $\cL_\infty$ and $\cL$ in Section \ref{seclinearopL}. 

\section{The linearized operator around the static N\'eel wall's phase}
\label{seclinearopL}

Let $\cL$ be the $L^2$-linearization of the $\nabla\cE$ around the static N\'eel wall's phase $\brt$. In this section, we examine the properties of the operator $\cL$ and locate its spectrum. In addition, we analyze the auxiliary asymptotic operator $\opli$, and  show that the difference $\cL_\infty-\cL:H^2\to L^2$ is a continuous and compact operator. 

The main results in this section are the following:


\begin{theorem}
\label{main_res_L}
    Let $\brt$ be the N\'eel wall phase and the operator $\cL:L^2\to L^2$ with domain $D(\cL)=H^2$ given by   
    \begin{equation}
        \label{defL0}
        \cL u := - \partial^2_xu + \mathcal{S}u - c_\theta u, \qquad u \in D(\cL),
    \end{equation}
    where $\cS:L^2\to L^2$ is a nonlocal operator, with domain $D(\cS)=H^1$, defined as
\begin{equation}
\label{defS}
 \cS u := \sin \brt(1+(-\Delta)^{1/2}) (u\sin \brt), \quad u \in D(\cS),
\end{equation}
and 
\begin{equation}
\label{defctheta}
 c_\theta := \cos \brt (1+(-\Delta)^{1/2}) \cos \brt.
\end{equation}
Then, $\cL$ is a closed densely defined self-adjoint linear operator whose $L^2$-spectrum satisfies  $$\sigma(\cL)\subset \{0\}\cup [\Lambda_0,\infty),$$ for some fixed positive constant $\Lambda_0$.
\end{theorem}

\begin{theorem}
\label{main_res_Linf}
    Let $\cL_\infty:L^2\to L^2$ with domain $D(\cL_\infty)=H^2$ be given by   
    \begin{equation}
        \label{defLinf}
        \cL_\infty u := - \partial^2_xu + (1+(-\Delta)^{1/2}) u.
\end{equation}
Then, $\cL_\infty$ is a closed densely defined self-adjoint invertible linear operator whose $L^2$-spectrum satisfies  
$$\sigma(\cL_\infty)\subset [1,\infty).$$
\end{theorem}
\begin{theorem}
\label{main_res_L-Linf}
    The linear operator $\cL_\infty-\cL:H^2\to L^2$ is continuous and compact. 
\end{theorem}

Theorem \ref{main_res_L-Linf} is a consequence of the Kolmogorov-Riesz Theorem \cite{Kolm31},\cite{Riesz33} that establishes the equivalence between the $L^p$-equicontinuity and $L^p$-uniform decaying, and the precompactness (or totally boundedness) on bounded sets in $L^p$ (Theorem \ref{themKolmogorov}). 
The $L^2$-equicontinuity of the image of $H^2$-bounded sets under $\cL_\infty-\cL$ follows from the continuity of the operator from $H^2$ to $H^1$. The $L^2$-uniform decaying is more complex. 


By introducing a smooth function $g_\delta$ that matches asymptotic behaviors of $\sin\brt$ at plus and minus infinity, we split $\opli-\cL$ as the sum of operators whose coefficients either have compact support or vanish at $\pm \infty$. For most local operators, the latter is enough to prove the $L^2$-uniform decaying property. However, since $\opli-\cL$ involves the nonlocal operator $(-\Delta)^{1/2}$ we require an extra step to show the  $L^2$ decaying property of the image of $H^2$-bounded sets under $[g_\delta,\cH]\partial_x$, where $\cH$ is the Hilbert transform, see Proposition \ref{propononlocal}.


In the Harmonic Analysis literature, it is well-known that the commutator $[f,K]:L^p(\R^n)\to L^p(\R^n)$ is compact for $1<p<\infty$ if $K$ is a Calder\'on-Zygmund integral operator with smooth kernel, and $f\in\cup_{q>1}L^q_{loc}(\R^n)$ such that $f$ has continuous mean oscillation $(CMO(\R^n))$\footnote{In some context\cite{LL22}, $CMO(\R^n)$ is also called $VMO(\R^n)$. However, this notion differs from the original definition \cite{Sa75,DM14,Da02}.}, see \cite{Uc78}. 
In our case, the proof of compactness of $[g_\theta, \cH]\partial_x: H^2\to H^1$ does not follow from the continuity of $\partial_x: H^2\to H^1$ and the the latter result, because $g_\theta$ does not belong to $CMO(\R^n)$. 

In subsection \ref{subsec:L_properties} below, we introduce the basic properties of operators $\cL$ and $\cS$, and in subsection \ref{subsec:L_spec}, we determine explicit bounds for the spectrum of $\cL$ and prove Theorem \ref{main_res_L}. Subsection \ref{subsec:asymptoticL} is devoted to analyzing the auxiliary operator $\cL_\infty$ and the prove of Theorem \ref{main_res_Linf}. In subsection~\ref{subsec:relcompactness}, we prove Theorem \ref{main_res_L-Linf} that, despite being key to showing the spectral stability of operator $\cA$, we believe it is important on its own because it gives an insight into the application of  Weyl's essential spectrum theorem in the context of non-local operators, that up to our knowledge is new.




\subsection{Basic properties}

We begin this section with the definitions of relative boundedness and relatively compact perturbations.  
\begin{definition}\label{initial_def} Let $\cP$ and $\cS$ be two linear operator on a Banach space $X$. We say that: 
\begin{enumerate}[label=(\roman*)]
    \item \label{def_relative_bound} The operator $\cS$ is relatively bounded with respect to $\cP$ (or $\cP$-bounded) if $D(\cP)\subset D(\cS)$ and 
    for every $u\in D(\cP)$
\[
\|\cS u\| \leq a\|u\|+b\|\cP u\|, 
\] 
where $a, b$ are non negative constants. The greatest lower bound $b_0$ of all possible constants $b$ is called the relative bound of $\cS$ with respect to $\cP$ (or the $\cP$-bound of $\cS$).

\item \label{def_relative_perturbation} The linear operator $\cS$ is a relatively compact perturbation of $\cP$ if for some $\lambda \in \res{\cP}$ the operator $(\cP-\cS)(\lambda \Id - \cP)^{-1}:X\to X$ is compact.
 
\end{enumerate}\end{definition}
 
%

\label{subsec:L_properties}

For a function $u\in L^2$, we define the Half-laplacian in terms  of the Fourier transform as follows, 
\[
(-\Delta)^{1/2} u := (|\xi|\widehat{u}(\xi))^\vee.
\]
Reagarded as an operator from $H^1$ to $L^2$, the Half-laplacian is a bounded linear operator.

From harmonic analysis, we know that the Half-laplacian is related to the Hilbert transform. The precise relationship is presented in the following lemma, whose proof can be found in many references(see, e.g., \cite{CCH21,King-v1-09,Ner75}). 


\begin{lemma}
\label{lemma:Hilbert}
    Let $\cH : L^2\to L^2$ be the Hilbert transform given by
    \[
    u\mapsto  \PV \frac{1}{\pi}  \int_\R \frac{u(s)}{x-s} \, ds.
    \]
    Then, $\cH$ is an isometry on $L^2$. Moreover, if $u\in H^1$ we have that 
    \[
    (-\Delta)^{1/2}u= \cH(\partial_x u) = \partial_x \cH u.
    \]
\end{lemma}

The following proposition summarizes the basic properties of the linearized operator $\cL$ and the N\'eel wall's phase, which have already been proved in \cite{CMO07}.

\begin{proposition}
\label{propL0}
The operator $\cL$ and the static N\'eel wall's phase $\brt$ satisfy:
 \begin{itemize}
\item[(a)]\label{propaa} $\partial_x \brt \in D(\cL)$ with $\cL \partial_x \brt = 0$. 
 \item[(b)]\label{propbb} For all $f \in L^2$ such that $f \perp \partial_x \brt$ in $L^2$ there exists a solution $u \in H^2$ to the equation $\cL u = f$. The solution is unique up to a constant multiple of $\partial_x \brt$.
\item[(c)]\label{propcc} There exists a uniform constant $\Lambda_0 > 0$ such that if $u \in H^1$ and $\langle u, \partial_x \brt \rangle_{L^2} = 0$, then
\begin{equation}
 \label{L0bound}
 \langle \cL u, u\rangle_{L^2} \geq \Lambda_0 \|u\|_{L^2}^2.
\end{equation}
\item[(d)]\label{propdd} Let $f \in \{\partial_x \brt\}^\perp \subset L^2$. Then the equation $\cL u = f$ has a strong solution $u \in H^2$, unique up to a constant multiple of $\partial_x \brt$. Moreover, if $u \in \{\partial_x \brt\}^\perp$, then
\begin{equation}
 \label{H2bound}
 \|u\|_{H^2} \leq C \|f\|_{L^2},
\end{equation}
for some $C > 0$.
 \end{itemize}
\end{proposition}
\begin{proof}
The proof follows from Lemmata 4 and 5, together with Proposition 1 in \cite{CMO07}. 
\end{proof}

For notational convenience, we denote
$s_{\theta} := \sin \brt$.
It can be shown\cite{CMO07}, that $s_\theta$ and $c_\theta$ both are real, smooth and bounded for in all $\R$ with $c_\theta \in H^2$. Moreover, 
\begin{corollary}
\label{corcsbded}
There exists a uniform constant $C > 0$ such that  
\begin{equation}
\label{boundscs}
\|c_\theta\|_{W^{1,\infty}}<C, \qquad  \|s_\theta\|_{W^{2,\infty}}<C.
\end{equation}
\end{corollary}
\begin{proof}
The proof follows directly from Corollary \ref{corthetabded} and the regularity of $c_{\theta}$ and $s_{\theta}$.
\end{proof}

The following lemma shows that the nonlocal operator $\cS$ is symmetric and its is related with the sesquilinear form  \eqref{defbilinearB}. 
 
\begin{lemma}\label{cor:operatorS} Let $\cS:L^2\to L^2$ with $D(\cS) = H^1$ be defined as in \ref{defS}. Then, $\cS$ is a symmetric operator and 
\[\pld{\cS u}{v} = b[s_{\theta}u,s_{\theta}v],
\]
for every $u,v \in D(\cS)$.
\end{lemma}
\begin{proof}Let $u, v \in H^1$. 
By Plancherel's theorem, we have that
\[
 \begin{aligned}
  \langle \cS u, v \rangle_{L^2} 
  = \int_\R (1 + |\xi|) \widehat{(s_{\theta} u)}(\xi) \widehat{(s_{\theta} v)}(\xi)^* \, d\xi
  = b[s_{\theta}u,s_{\theta}v].
 \end{aligned}
\]
 The symmetry of $\cS$ follows because  $H^1\subset L^2$ densely and by the hermiticity of $b$, we get that
\[
\pld{\cS u}{v} = b[s_{\theta}u, s_{\theta}v] =b[s_{\theta}v, s_{\theta}u]^* =\pld{\cS v}{u}^* =\pld{u}{\cS v}.
\]
\end{proof}


We finish this subsection by showing  self-adjointness of the operator $\cL$. 


\begin{lemma}
\label{thmLselfadj}
The operator $\cL : L^2 \to L^2$ with domain $D(\cL) = H^2$ is self-adjoint and closed.
\end{lemma}
\begin{proof}
First, note that $\cL$ is clearly a symmetric operator, because its domain is dense in $L^2$ and there holds
\[
\begin{aligned}
\langle \cL u, v \rangle_{L^2} = 
\langle  u, - \partial_x^2 v \rangle_{L^2} + \langle  u, \cS v \rangle_{L^2} + \langle  u, c_{\theta} v \rangle_{L^2}
= \langle u, \cL v \rangle_{L^2},
\end{aligned}
\]
for all $u, v \in H^2$ which follows 
by an integration by parts, an application of Lemma \ref{cor:operatorS} and the fact that $c_{\theta}$ is real.

It is well-known that for every $u \in H^2$ there holds the estimate 
\begin{equation}
\label{Katoineq}
\| \partial_xu \|_{L^2} \leq k \| \partial^2_xu \|_{L^2} + \frac{2}{k} \| u \|_{L^2},
\end{equation}
for any arbitrary $k > 0$ (see Kato \cite{Kat80}, p. 192). Let 
\[
 \left\{
\begin{aligned}
&\tcS : L^2 \to L^2,\\
&D(\tcS) = H^1,\\
& \tcS u := s_{\theta} (-\Delta)^{1/2} (s_{\theta} u), \quad u \in D(\tcS),
 \end{aligned}
 \right.
\]
so that $\cS = s_{\theta}^2 \Id + \tcS$. Following the arguments of Lemma \ref{cor:operatorS}, it follows that $\tcS$ is a symmetric operator. By Corollary \ref{corcsbded}, 
there exists a constant $C_0 > 0$ such 
$\|\partial_xs_{\theta} \|_\infty \leq C_0$. Therefore,
\[
\begin{aligned}
\| \tcS u \|_{L^2} 
&\leq \left( \int_\R | (-\Delta)^{1/2} (s_{\theta}(x) u) |^2 \, dx \right)^{1/2} = \left( \int_\R |\xi|^2 | \widehat{(s_{\theta} u)}(\xi) |^2 \, d\xi \right)^{1/2}\\
&\leq \| \partial_x( s_{\theta} u ) \|_{L^2} \leq \| \partial_xs_{\theta} \|_\infty \| u \|_{L^2} + \| s_{\theta} \|_\infty \| \partial_xu \|_{L^2} \leq C_0  \| u \|_{L^2} + \| \partial_xu \|_{L^2},
\end{aligned}
\]
and by  \eqref{Katoineq} we get
\[
\| \tcS u \|_{L^2} \leq k \| -\partial^2_xu \|_{L^2} + \Big( C_0 + \frac{2}{k}\Big) \| u \|_{L^2},
\]
for all $u \in H^2$ and any arbitrary $k > 0$. Because $D(-\partial_x^2) = H^2 \subset D(\tcS) = H^1$ 
we conclude that the symmetric operator $\tcS$ is relatively bounded with respect to $-\partial_x^2$, and its relative bound is zero. Consequently, we may apply Kato-Rellich's theorem (see \cite{ReedS2}, Theorem X.12, p. 162) to conclude that the operator $\tcS - \partial_x^2 : L^2 \to L^2$ with domain $D( \tcS - \partial_x^2) = D(- \partial_x^2) = H^2$ is self-adjoint.

Now, let us write $\cL = - \partial_x^2 + \cS - c_{\theta} \Id = - \partial_x^2 + \tcS + \beta \Id$, where $\beta := s_{\theta}^2 - c_{\theta}$ is a bounded real smooth coefficient. Clearly,
\[
\| \beta u \|_{L^2} \leq \| \beta \|_\infty \| u \|_{L^2} \leq \| \beta \|_\infty \| u \|_{L^2} + k \| (\tcS - \partial_x^2) u \|_{L^2},
\]
for all $u \in H^2$ and for any $k > 0$. Since $D(\tcS - \partial_x^2) = H^2 \subset D(\beta \Id) = L^2$, we conclude that the symmetric operator $\beta \Id$ is $(\tcS - \partial_x^2)-$bounded with relative bound equal to zero. Upon application of Kato-Rellich's theorem, we conclude that the operator $\cL = - \partial_x^2 + \tcS + \beta \Id$ with domain $D(\cL) = H^2$  is self-adjoint. 

Finally, the closeness of $\cL$ follows because every self-adjoint operator is closed. 
\end{proof}

\subsection{The spectrum of $\cL$}
\label{subsec:L_spec}
The operator $\cL$ is a self-adjoint operator by Lemma \ref{thmLselfadj}, and its $L^2$-spectrum is real. From Proposition \ref{propL0}, $\partial_x \brt \in L^2$ is an eigenfunction of $\cL$ associated with the eigenvalue $\lambda = 0$. 
Moreover, the kernel of $\cL$ is spanned by $\partial_x \brt$ and has dimension one. Otherwise, there exists $u \in H^2 \cap \ker (\cL)$, $u \neq 0$, with $u = u_\perp + \alpha \partial_x \brt$ for some $\alpha \in \C$ where $\pld{u_\perp}{\partial_x \brt}=0$. By Proposition \ref{propL0} \hyperref[propcc]{\rm{(c)}}, 
\[
0 = \langle \cL(u_\perp + \alpha \partial_x \brt), u_\perp + \alpha \partial_x \brt \rangle_{L^2} = \langle \cL u_\perp, u_\perp \rangle_{L^2} \geq \Lambda_0 \|u_\perp\|_{L^2}^2,
\]
so that $u_\perp = 0$, a contradiction.  Therefore, the geometric multiplicity of $\lambda = 0$ is one. Since in a Hilbert, the algebraic and geometric multiplicities coincide  for self-adjoint  (see Kato \cite{Kat80}, p. 273), and we get the following result: 


\begin{lemma}
\label{corzeroL0}
 $\lambda = 0$ is a simple eigenvalue of the operator $\cL: L^2 \to L^2$, with eigenfunction $\partial_x \brt \in D(\cL) = H^2$. Moreover, $0\in\ptsp(\cL)$.
\end{lemma}


Now, we give:

\begin{proof}[\bf Proof of Theorem \ref{main_res_L}]
By Theorem \ref{thmLselfadj} $\cL$ is a closed densely defined self-adjoint linear operator. It remains to localize its $L^2$-spectrum.

Let $L^2=\lspan{\{\partial_x \brt\}}\oplus L^2_\perp$ where $L^2_\perp$ is $\partial_x\brt$'s $L^2$-orthogonal complement.
 The operator $\cL$ is self-adjoint; hence its spectrum is real. In addition, by Lemma \ref{corzeroL0}, $\lambda = 0$ is an isolated simple eigenvalue, and  from 
the spectral decomposition theorem  $\sigma(\cL)_{|L^2_\perp} = \sigma(\cL) \backslash  \{0\}$ (see Theorem III-6.17, p. 178, in Kato \cite{Kat80}).

Now, we claim that
\[
 \left\{
 \begin{aligned}
&\cL_{|L^2_\perp} : L^2_\perp \to L^2_\perp,\\
&D(\cL_{|L^2_\perp}) = D(\cL) \cap L^2_\perp = H^2 \cap L^2_\perp,\\
& \cL_{|L^2_\perp} u := \cL u, \quad u \in D(\cL_{|L^2_\perp}).
 \end{aligned}
 \right.
 \]
is also self-adjoint.
 
 Clearly, $\cL_{|L^2_\perp}$ is symmetric because $\cL$ is symmetric and $D(\cL_{|L^2_\perp})\subset L^2_\perp$ densely. 
 In order to show that $\cL_{|L^2_\perp}$ is self-adjoint it suffices to show that $(\cL_{|L^2_\perp} \pm i)(D(\cL) \cap L^2_\perp) = L^2_\perp$ (see, e.g., Theorem VIII.3, p. 256, in \cite{ReedS1}). Let $\cP_0$ the orthogonal projection from $L^2$ to $\lspan \{\partial \brt\}$. We already know that $\cL \pm i : D(\cL) \to L^2$ is surjective because $\cL$ is self-adjoint, and $\cL$ and $\cP_0$ commute. Therefore, for $v \in L^2_\perp$ there exist elements $u_\pm \in D(\cL) = H^2$ such that $(\cL \pm i ) u_\pm = v$. Thus,  $(\Id - \cP_0) u_\pm \in L^2_\perp$, and
 \[
 (\cL \pm i)(\Id - \cP_0) u_\pm 
 = v - \cP_0 (\cL \pm i ) u_\pm\\= (\Id - \cP_0) v = v
 \]
That is, $(\cL_{|L^2_\perp} \pm i ) : D(\cL) \cap L^2_\perp \to L^2_\perp$ is surjective, and $\cL_{|L^2_\perp}$ is self-adjoint as claimed.
 
 Finally, from Rayleigh's formula for semi-bounded self-adjoint operators (\cite{Kat80}, p. 278) and Proposition \ref{propL0} \hyperref[propcc]{\rm{(c)}}, we have 
 \[
 \langle \cL_{|L^2_\perp} u, u\rangle_{L^2} = \langle \cL u, u\rangle_{L^2} \geq \Lambda_0 \|u\|_{L^2}^2,
 \]
 for all $u \in D(\cL) \cap L^2_\perp = H^2 \cap \{ \partial_x \brt \}^{\perp}_{L^2}$.  Therefore, $\sigma(\cL_{L^2_\perp}) \subset [\Lambda_0, \infty)$ and Kato's decomposition theorem yields $\sigma(\cL)_{L^2} \subset \{ 0 \} \cup [\Lambda_0, \infty)$ as claimed.

\end{proof}

\subsection{The asymptotic operator $\cL_\infty$}
\label{subsec:asymptoticL}
This subsection examines the operator $\opli$ defined in Theorem $\ref{main_res_L-Linf}$. By replacing $\cL$'s coefficients with their limits at $x \to \pm \infty$, we obtain two limiting operators $\cL_{+\infty}$ and $\cL_{-\infty}$. In the present case,  $\cL_{+\infty}=\cL_{-\infty}=\opli$. 

Let $a_\infty[\cdot, \cdot] : H^1 \times H^1 \to \C$ be the sesquilinear form 
\begin{equation*}
a_\infty[u,v] 
:= \langle \partial_xu, \partial_xv \rangle_{L^2} + b[u,v],
\end{equation*}
where $b[\cdot,\cdot]$ is as  in \eqref{defbilinearB}.
For $f\in L^2$, the weak formulation of the nonhomogeneous linear problem  
\begin{equation}
\label{equLinfty}
\cL_\infty u = f,
\end{equation} 
is given by 
$a_\infty[u,v] = \langle \cL_\infty u, v \rangle_{L^2} = \langle f, v \rangle_{L^2}$ 
for every $v \in H^1$.  In terms of the Fourier transform, it is straightforward that 
the bilinear form $a_\infty[\cdot,\cdot]$ is uniformly elliptic and continuous in $H^1$.

\begin{lemma}
\label{lemsolLinf}
For every $f \in L^2$ there exists a unique solution $u \in H^2$ to \eqref{equLinfty}. Moreover, 
\begin{equation}
\label{estufH2}
\| u \|_{H^2} \leq \| f \|_{L^2}.
\end{equation}
\end{lemma}
\begin{proof}
Since $a_\infty[\cdot,\cdot]$ is uniformly elliptic and continuous in $H^1$, for each $f \in L^2$ there exists a unique weak solution $u \in H^1$ to \eqref{equLinfty} by the Lax-Milgram theorem. Hence,
\[
\langle \partial_xu, \partial_x\varphi \rangle_{L^2} + \langle (1+(-\Delta)^{1/2})u, \varphi \rangle_{L^2} = \langle f, \varphi \rangle_{L^2}.
\]
for any test function $\varphi \in C_0^\infty(\R)$, and by Plancherel's identity we get
\[
\int_\R \big[ (1 + |\xi| + \xi^2) \widehat{u}(\xi) - \widehat{f}(\xi) \big] \widehat{\varphi}(\xi)^* \, d\xi = 0,
\]
 for all $\varphi \in C_0^\infty(\R)$. Therefore, $(1 + |\xi| + \xi^2) \widehat{u}(\xi) = \widehat{f}(\xi)$ a.e. in $\xi \in \R$, and
\[
\| u \|_{H^2}^2 = \int_\R (1 + \xi^2)^2 |\widehat{u}(\xi)|^2 \, d\xi = \int_\R \Big( \frac{1+\xi^2}{1+|\xi|+\xi^2}\Big)^2 |\widehat{f}(\xi)|^2 \, d\xi \leq \|f \|_{L^2}^2.
\] 
\end{proof}

Now, we prove Theorem \ref{main_res_Linf}
\begin{proof}[\bf Proof of Theorem \ref{main_res_Linf}]
By definition, $\cL_\infty$ is densely defined in $L^2$, and its invertibility follows from Lemma \ref{lemsolLinf}. As in the proof of Lemma \ref{thmLselfadj}, the self-adjointness of $\opli$ follows by Kato-Rellich Theorem, and $\sigma(\cL_\infty) \subset \R$. Moreover, $\cL_\infty$ is semi-bounded since  $\langle \cL_\infty u, u \rangle_{L^2} = a_\infty[u,u] \geq \| u \|_{H^1}^2 \geq \| u \|_{L^2}^2$ for $u \in D(\cL_\infty) = H^2$. Finally, using Lemma \ref{lemsolLinf}, we get
\[
\inf_{0 \neq v \in D(\cL_\infty)} \frac{\langle \cL_\infty v, v \rangle_{L^2}}{\| v \|_{L^2}^2} = \inf_{0 \neq v \in D(\cL_\infty)} \frac{a_\infty[v,v]}{\| v \|_{L^2}^2} \geq 1,
\] 
and by Rayleigh's spectral bound for semi-bounded self-adjoint operators in Hilbert spaces (cf. \cite{Kat80}, p. 278) we conclude that $\inf \sigma(\cL_\infty)>1$.
\end{proof}

\begin{remark}
\label{lemressol}
By definition of the resolvent set and Theorem \ref{main_res_Linf}, for any 
$\mu \in \C \backslash [1,\infty) \subset \rho(\cL_\infty)$, there exists $u\in H^2$ solution of $(\cL_\infty - \mu) u = f$. Moreover, arguing as in the proof of Lemma~\ref{lemsolLinf}, there exists a constant $C=C(\eta)>0$ such that 
\begin{equation}
\label{H2est}
\| u \|_{H^2} \leq C(\mu) \| f \|_{L^2}.
\end{equation}
\end{remark}

\subsection{Relative compactness}
\label{subsec:relcompactness} The goal of this subsection is the proof of Theorem \ref{main_res_L-Linf}.
We begin by establishing the continuity of $\opli-\opl$ as an operator from $H^2$ to $H^1$.

\begin{lemma}
\label{themcontinuity}
$\cL_\infty - \cL$ continuously maps $H^2$ into $H^1$.
\end{lemma}
\begin{proof}
Let $u \in H^2$ then, we have that 
\begin{equation}
\label{diff_L}
(\cL_\infty - \cL) u = (1+(-\Delta)^{1/2})u - s_{\theta}(1+(-\Delta)^{1/2})(s_{\theta}u) + c_{\theta} u.
\end{equation}
By Corollary \eqref{corcsbded}, we have the bounds
\[
\| c_{\theta} u \|_{H^1}^2 
\leq 2\|c_\theta\|^2_{W^{1,\infty}} \| u\|_{H^1}^2\leq C\| u\|_{H^2}^2,
\]
for some $C > 0$. Moreover,
\[
\begin{aligned}
\| (1+(-\Delta)^{1/2})u \|_{H^1}^2 &= \int_\R (1+\xi^2) \big| \big((1+(-\Delta)^{1/2})u\big)^{\wedge}(\xi) \big|^2 \, d \xi \\
&\leq 2 \int_\R (1+\xi^2)^2 |\widehat{u}(\xi)|^2 \, d \xi = 2 \| u\|_{H^2}^2.
\end{aligned}
\]
From the later inequality and Corollary \eqref{corcsbded}, there exists $\widetilde{C} > 0$ such that
\[
\| s_{\theta}(1+(-\Delta)^{1/2})(s_{\theta}u) \|_{H^1}^2 
\leq 2 \|s_\theta\|^2_{W^{1,\infty}} \| (1+(-\Delta)^{1/2})(s_{\theta}u) \|_{H^1}^2
\leq \widetilde{C} \| u \|_{H^2}^2.
\]
We combine all  estimates above to conclude that there exists a constant $C > 0$ such that
\begin{equation}
\label{estLLinf}
\| (\cL_\infty - \cL) u \|_{H^1} \leq C \| u \|_{H^2},
\end{equation} 
for all $u \in D(\cL) = H^2$. This shows the result.
\end{proof}

At this point, we cite two theorems, one due to Kolmogorov \cite{Kolm31} and Riesz \cite{Riesz33} (see, for example, \cite{HaHe10} and the references therein) and the other due to Pego \cite{P85a},  describing totally bounded sets in $L^p$ and $L^2$, respectively. We recall that precompactness and totally boundedness are equivalent concepts for sets contained in complete metric spaces.

\begin{theorem}[Kolmogorov-Riesz \cite{Kolm31,Riesz33}] 
\label{themKolmogorov} 
A bounded set $\cF\subset L^p(\R^n)$ with $1\leq p<\infty$ is totally bounded if and only if  
\begin{itemize}
    \item[(a)] {\em ($L^p$-equicontinuity)} $\lim_{h\to 0}\int_{\R^n}|u(x+h)-u(x)|^p\,dx = 0$ uniformly for $u\in \cF$, and 
    \item[(b)] {\em ($L^p$-uniform decay)} $\lim_{R\to\infty}\int_{|x|>R}|u(x)|^p\,dx = 0$ uniformly for $u\in \cF$.
\end{itemize}
\end{theorem} 

\begin{theorem}[Pego \cite{P85a}] 
\label{themPego} 
Let $\cF$ be a bounded set of $L^2(\R^n)$ and $\widehat{\cF}:=\{\widehat{u}\,|\, u\in \cF\}$. The functions for $\cF$ are $L^2$-equicontinuous 
if and only if the functions for $\widehat{\cF}$ decay uniformly in $L^2$ and vice versa.
\end{theorem}

Before using Theorems \ref{themKolmogorov} and \ref{themPego} in our proofs, we need the following preparatory results. 

\begin{proposition} 
\label{proptotbon}
Let $\cF$ be a bounded set in $H^1$ and  $\phi\in H^1$. Then the set $\phi\cF:=\{\phi u \,|\, u\in \cF\}$ is precompact in $L^2$.
\end{proposition}
\begin{proof}
    Let $\{v_k\}_{k\in \N}\subset \phi\cF$. By definition, there exists a bounded sequence $\{u_k\}_{k\in\N}\subset \cF$ such that $v_k =\phi u_k$ for each $k\in \N$.  The sequence, $\{u_k\}_{k\in\N}$ contains an $H^1$-weakly convergent subsequence $\{v_{k_n}\}_{n\in\N}$ that, without loss of generality, has limit $u=0$. 

    We claim that $v_{k_n}=\phi u_{k_n}\to 0$ strongly in $L^2$. Let $\ep>0$, $\eta\in C_0^\infty$ be a standard cut-off function such that $\supp \eta \subset B_1$, and $\eta_R(x) := \eta(x/R)$. Now, we split $u_{k_n}$ as the sum of $u_{k_n}\eta_R$ and $u_{k_n}(1-\eta_R)$. Because $\phi\in L^2$ and $u_{k_n}$ is uniformly bounded in $L^\infty$ (by Sobolev inequality), there exists $R>0$ such that for every $k\in \N$,
    \[
    \nld{\phi (1-\eta_R)u_{k_n}}<\ep/2.
    \]
    Moreover, $H^1$-weak convergence implies locally uniform convergence, therefore there exists $N_0\in \N$ such that for each $n>N_0$ 
    \[
    \nld{\phi \eta_R u_{k_n}}<\ep/2.
    \]
    Combining both inequalities $v_{k_n}\to 0$ strongly in $L^2$, and the proof is complete.
\end{proof}

In the proof of the following results, if a set $\cK$ belongs to an operator $\cG$'s domain, we denote by $\cG\cK$ the image of $\cK$ under the action of $\cG$.

\begin{proposition} \label{propononlocal}
    Let $\delta>0$, $g_\delta\in H^1\cap C^\infty$ be a bounded monotonic function such that $\supp\partial_x g_\delta \subset [-\delta,\delta]$, and  $\cH:L^2\to L^2$ be the Hilbert transform. Then, the linear operator $[g_\delta, \cH]\partial_x:H^2\to H^1$ is compact in $L^2$.
\end{proposition}

\begin{proof}
The result follows from Theorem~\ref{proptotbon} by showing that for any bounded $\cF\subset H^2$ the set $[g_\delta,\cH]\partial_x \cF$ is $L^2$-equicontinuous and uniformly $L^2$-decaying. 
 We present the proof in two steps.


First, let $\cG\subset H^1$ be bounded. Then, for any $v\in\cG$, it holds that 
 \[
\begin{aligned}
\int_{\{|\xi|>R\}}|\hat{v}(\xi)|^2 \, d\xi 
    & \leq \frac{1}{1+R^2}\int_{\R}(1+\xi^2)|\hat{v}(\xi)|^2\, d\xi = \frac{\nhu{v}^2}{1+R^2} \leq \frac{M^2}{1+R^2}.
\end{aligned}
\]
Therefore, $\widehat{\cG}$ is $L^2$-uniformly decaying, and $\cG$ is $L^2$-equicontinuous by Theorem \ref{themPego}.

Thus, the  $L^2$-equicontinuity 
 of $[g_\delta,\cH]\partial_x \cF$ follows by  showing its boundedness in $H^1$.
For any $u\in \cF$, we get that
\[
\nld{[g_\delta,  \cH ]\partial_xu}^2 \leq \nld{g_\delta  \cH \partial_xu}^2 + \nld{ \cH (g_\delta\partial_xu)}^2 \leq 2\|g_\delta\|_{W^{1,\infty}}^2\nld{\partial_xu}^2,
\]
by the continuity of $\cH$ in $L^2$. Analogously, we obtain the estimate   
 \[
\nld{\partial_x([g_\delta,  \cH ]\partial_xu)}^2 \leq 
\nld{(\partial_xg_\delta)  \cH \partial_xu}^2+\nld{g_\delta  \cH \partial^2_xu}^2 + \nld{ \cH \partial_x(g_\delta\partial_xu)}^2\\
\leq  2\|g_\delta\|_{W^{1,\infty}}^2\nhu{\partial_xu}^2.
\]
 Therefore, $\nhu{[g_\delta,  \cH ]\partial_xu} \leq  C\nhd{u}$ for some positive constant $C$, and we conclude by the boundedness of $\cF$ in $H^2$.

 Second, we show the $L^2$-uniform decay of $[g_\delta,\cH]\partial_x \cF$.
 Let $R>2\delta$ be a fixed constant. We claim that 
 \begin{equation}
\label{eq:commutator}
\pi[g_\delta,  \cH ]\partial_xu(x) = C\int^{\infty}_{\delta} \frac{\partial_xu(-\sgn(x)y)}{y+|x|} \, dy + \int_{-\delta}^{\delta} \frac{g_\delta(y)-g_\delta(\sgn(x)\delta)}{y-x} \, \partial_xu(y) \, dy,
\end{equation}
for $|x|>R$.

We postpone the proof of the claim and prove the result.
By integration by parts on former term on the right hand side of \eqref{eq:commutator} we get
\[
\begin{aligned}
\int^{\infty}_{\delta} \frac{\partial_xu(-\sgn(x)y)}{y+|x|} \, dy 
= & -\sgn(x)\int^{\infty}_{\delta} \frac{u(-\sgn(x)y)-u(-\sgn(x)\delta)}{(y+|x|)^2} \, dy.
\end{aligned}
\]
Thus, if $M>0$ is a uniform bound for $\cF$, then  $\|u\|_{L^\infty}\leq M$ for every $u\in \cF$, and
\[
\left|\int^{\infty}_{\delta} \frac{\partial_xu(-\sgn(x)y)}{y+|x|} \, dy\right|
\leq 2M\int^{\infty}_{\delta} \frac{1}{(y+|x|)^2} \, dy = \frac{2M}{\delta +|x|}.
\]
By integration on $x$ over $\R\setminus [-R,R]$, we get
\begin{equation}
\label{eq:commutator_p1}
\int_{|x|>R}\left(\int^{\infty}_{\delta} \frac{\partial_xu(-\sgn(x)y)}{y+|x|} \, dy\right)^2 dx
\leq 4M^2\int_{|x|>R}\frac{dx}{(\delta+|x|)^2} \leq \frac{8M^2}{\delta+R}.
\end{equation}
Hence,  the former term on the right hand side of \eqref{eq:commutator} is $L^2$-uniformly decaying. 

Now we consider the latter term on the right hand side of  \eqref{eq:commutator}.
Let $C_0:=|g_\delta(-\delta)|+|g_\delta(\delta))|$. By recalling that $g_\delta$ is monotone, from Jensen's inequality, we get 
\[
\begin{aligned}
\int_{|x|>R}\left(\int_{-\delta}^{\delta} \frac{g_\delta(y)-g_\delta(\sgn(x)\delta)}{y-x} \, \partial_xu(y) \, dy\right)^2 \! dx &\leq 2\delta \int_{|x|>R}\int_{-\delta}^{\delta} \frac{\left(g_\delta(y)-g_\delta(\sgn(x)\delta)\right)^2}{\left(y-x\right)^2} \, (\partial_xu(y))^2 \, dy dx\\
&\leq  8C_0^2\delta \int_{|x|>R}\int_{-\delta}^{\delta} \frac{1}{\left(y-x\right)^2} \, (\partial_xu(y))^2 \, dy dx.
\end{aligned}
\]
Since $\partial_xu\in L^2$ and $(y-x)^2 \geq (|x|-\delta)$ for every $y\in(-\delta,\delta)$, we obtain
\begin{equation}
\label{eq:commutator_p2}
\rint_{|x|>R}\left(\int_{-\delta}^{\delta} \frac{g_\delta(y)-g_\delta(\sgn(x)\delta)}{y-x} \, \partial_xu(y) \, dy\right)^2 dx\leq   \rint_{|x|>R}\frac{8C_0^2\delta\nld{\partial_xu}^2 dx}{\left(|x|-\delta\right)^2}
\leq  \frac{16C_0^2\delta M^2}{R-\delta}.
\end{equation}
Combining  Young's inequality with equations  \eqref{eq:commutator_p1} and \eqref{eq:commutator_p2}, we conclude that
\[
\begin{aligned}
\int_{|x|>R}([g_\delta,  \cH ]\partial_x u)^2 dx
&\leq \frac{16M^2(C^2+2C_0^2\delta)}{\pi(R-\delta)}.
\end{aligned}
\]
Therefore, $[g_\delta,\cH]\partial_x \cF$ is $L^2$-uniformly decaying, and the result follows.
\medskip

Finally, we prove claim \eqref{eq:commutator}. Let $v=\partial_x u$ for notational convenience. 
By the fundamental theorem of calculus and Lemma \ref{lemma:Hilbert}, 
\[
\pi[g_\delta,  \cH ]v(x) 
=  \lim_{\ep \to 0}\int_{|h|>\ep} \frac{g_\delta(x+h)-g_\delta(x)}{h}v(x+h) \, dh
=  \lim_{\ep \to 0}\int_{|h|>\ep} \frac{1}{h}\int_x^{x+h}g_\delta'(t)\,dt \, v(x+h) \, dh.
\]
Assuming that $x>R$ and letting $\ep<\delta$, the expression $\pi[g_\delta,  \cH ]v(x)$ becomes
\[
\begin{aligned}
\pi[g_\delta, \cH]v(x) = & \int_{-\infty}^{-x+\delta} \!\frac{1}{h}\int_{x}^{x+h} \!g_\delta'(t)\,dt \, v(x+h) \, dh + \\
& + \lim_{\ep \to 0} \left[\int_{-x+\delta}^{-\ep} \frac{1}{h}\int_{x}^{x+h} \!g_\delta'(t)\,dt \, v(x+h) \, dh + \int_{\ep}^{\infty}\frac{1}{h}\int_{x}^{x+h}\!g_\delta'(t)\,dt \, v(x+h) \, dh \right].
\end{aligned}
\]
In the latter equality, the last two integrals vanish since $\supp g' \subset [-\delta,\delta]$ and $\delta<x+h$ for $h>\delta-x$.  Furthermore, the first term is given by 
\[
\pi[g_\delta,  \cH ]v(x) = \int_{-\infty}^{-x-\delta} \frac{1}{h}\int_{x}^{x+h}g_\delta'(t)\,dt \, v(x+h) \, dh + \int_{-x-\delta}^{-x+\delta} \frac{1}{h}\int_{x}^{x+h}g_\delta'(t)\,dt \, v(x+h) \, dh.
\]
Since $\supp g'_\delta\subset [x+h,x]$ for $h\leq -x-\delta$, the integral   $\int_{|x|\leq \delta}g'(x)dx$ is independent of $h$. Hence,
\[
\pi[g_\delta,  \cH ]v(x) =-C\int_{-\infty}^{-x-\delta} \frac{v(x+h)}{h} \, dh + \int_{-x-\delta}^{-x+\delta} \frac{1}{h}\int_{\delta}^{x+h}g_\delta'(t)\,dt \, v(x+h) \, dh,
\]
 where $C=\int_{|x|\leq \delta}g'(x)dx$.  

Now, by letting $y=x+h$, the fundamental theorem of calculus yields 
\[
\begin{aligned}
\pi[g_\delta,  \cH ]v(x) = &-C\int_{-\infty}^{-\delta} \frac{v(y)}{y-x} \, dy + \int_{-\delta}^{\delta} \frac{1}{y-x}\int_{\delta}^{y}g_\delta'(t)\,dt \, v(y) \, dy,\\
= &-C\int_{-\infty}^{-\delta} \frac{v(y)}{y-x} \, dy - \int_{-\delta}^{\delta} \frac{g_\delta(\delta)-g_\delta(y)}{y-x} \, v(y) \, dy.\\
\end{aligned}
\]
Using an analogous argument for $x<-R$ we conclude \eqref{eq:commutator} and the proof is finished.

\end{proof}

Now, we prove this subsection's main result.

\begin{proof}[\bf Proof  of Theorem \ref{main_res_L-Linf}]
Let $\delta>0$ fixed,   
$g_\delta\in C^{\infty}$ be an increasing antisymmetric function such that $g_{\delta}(x)=x/|x|$ for $|x|\geq \delta$, $u\in H^2$, and denote $(1+(-\Delta)^{1/2})$ by $\cT$. Adding and substrating $g_{\delta}^2(x)\cT u+g_{\delta}(x)\cT(s_{\theta}u)+g_{\delta}(x)\cT(g_{\delta}(x)u)$ to $(\opli-\opl)u$, we recast the later as  
\begin{equation}
\label{recast_L}
(\cL_\infty - \cL) u = g_\delta[g_\delta, \cT] u+ \cQ_1 u + \cQ_2 u + \cQ_3 (s_{\theta}u)+\cQ_4 u,
\end{equation}
where 
$[g_\delta, \cT]$ denotes the commutator of $g_\delta \Id$ with $\cT$,  and 
\begin{align}
\label{aux_A12}
\cQ_1 u &:=[1-g^2_{\delta}]\cT u, &\cQ_2 u&:= g_{\delta}\cT[(g_{\delta}-s_{\theta})u],\\
\label{aux_A34}
\cQ_3 u&:=[g_{\delta}-s_{\theta}]\cT u, &\cQ_4 u&:= c_{\theta} u.
\end{align}

Now, we analyse the compactness of $[g_\delta,\cT]$ and the $\cQ$s operators. Let $\cF\subset H^2$ such that $\sup_{u\in \cF}\nhd{u}<M$ for some $M>0$. By the continuity of $\cT:H^2\to H^1$ (see proof in Lemma \ref{themcontinuity}) and Corollary~\ref{corthetabded}, we have that
\[
\nhu{c_\theta}^2\leq 2\|\cos \brt\|^2_{W^{1,\infty}}\nhu{\cT \cos \brt}^2\leq C\nhd{\cos \brt}^2,
\]
for some positive $C$. Hence,  $c_\theta\in H^1$ and  $1-g^2_\delta \in H^1$ by assumption. Therefore, $\cQ_1 \cF$ and $\cQ_4 \cF$ are precompact sets in $L^2$ by   Proposition \ref{proptotbon} and $\cQ_1$ and $\cQ_4$'s $L^2$-compactness follows. Analogously, $\cQ_3 \cF$ is precompact in $L^2$ since $g_\delta-s_\theta\in H^1$ and $\cT:H^2\to H^1$ is continuous. Thus, $\cQ_3:H^2\to H^1$ is compact in $L^2$. Let $\Id$ be the identity operator on $H^2$, then    
$s_\theta \Id$ is continuous on $H^2$, and the composition  $\cQ_3\circ s_\theta I:H^2\to H^1$ is compact in $L^2$. 
Because $g_\delta-s_\theta\in H^1$ and $H^2$ is continuously embedded into $H^1$,  $[g_\delta-s_\theta]\cF$ is precompact in $H^1$. Thus, $\cQ_2:H^2\to H^1$ is compact in  $L^2$ by the continuity of $g_\delta \cT$ from $H^1$ to $L^2$.
Since $g_\delta \Id:L^2\to L^2$
is continuous, the precompactness of $[g_\delta,\cT]\cF$ in $L^2$ implies  that $g_\delta[g_\delta,\cT]:H^2\to H^1$ is compact in $L^2$. Now, using Lemma~\ref{lemma:Hilbert}, we get 
\[
\begin{aligned}\relax
[g_\delta, \cT] = [g_\delta, (-\Delta)^{1/2}] = [g_\delta, \cH \circ \partial_x] 
&=  [g_\delta, \cH]\partial_x  - \cH\circ (\partial_xg_\delta)\Id.
\end{aligned}
\]
By the regularity of $g_\delta$ and Proposition~ \ref{proptotbon}, $(\partial_x g_\delta)\Id:H^2\to H^1 $ is compact in $L^2$. Therefore, 
$\cH\circ (\partial_xg_\delta)\Id:H^2\to H^1$  is compact in $L^2$  by $\cH$'s $L^2$-continuity. Now, by Proposition \ref{propononlocal}, $[g_\delta, \cH]\partial_x:H^2\to H^1$ is compact in $L^2$.

Because, the set of compact operators between Banach spaces is a linear manifold, the $L^2$-compactness for each term on the right hand side of $\eqref{recast_L}$ implies that  $(\opli-\opl):H^2\to H^1$ is compact in $L^2$. 
Finally, the continuity of the inclusion $H^1\xhookrightarrow{}L^2$ implies the compactness of $\opli-\opl:H^2\to L^2$ and the proof is completed.

\end{proof}

\begin{corollary}
The operator $\cL$ is a relatively compact perturbation of $\cL_\infty$ and they have the same essential spectrum. 
\end{corollary}
\begin{proof}
Let $\mu\in \res{\opli}$, hence $(\mu - \cL_\infty)^{-1} : L^2 \to H^2$ is a continuous linear operator and by Theorem \ref{main_res_L-Linf}, $\opli-\opl:H^2\to L^2$ is compact. 
Thus, the operator $(\opli-\opl)(\mu - \cL_\infty)^{-1}$ is compact on $L^2$ and by Weyl's essential spectrum theorem (see, e.g., \cite{KaPro13}, p. 29) the essential spectrum of  
$\cL$ and $\cL_\infty$ coincide. 


\end{proof}

\section{Perturbation equations and spectral stability}
\label{secspectral}

In this section, we pose the Cauchy problem \eqref{reddyneq} as the two-dimensional ODE system, and analyze its linearization $\cA$ around $(\brt,0)$. In the following result, we establish $\cA$'s closedness and stability by localizing its spectrum. 


\begin{theorem}
\label{mainspthm}
Let $\Lambda_0$ be as in Proposition \ref{propL0}, assume $\nu > 0$ fixed, and  define $\cA:H^1\times L^2\to H^1\times L^2$ with domain $D(\cA)=H^2\times H^1$ as 
\begin{equation}
\label{eq:A}
    \cA =\begin{pmatrix}0 & \Id \\ - \cL & - \nu \Id\end{pmatrix}.
\end{equation}
Then, 
\begin{equation}
\label{eq:spectralloc}
\sigma(\cA) \subset \{ 0 \} \cup \big\{ \lambda \in \C \, : \, \Re \lambda \leq - \zeta_0(\nu) < 0 \big\},
\end{equation}
where 
\begin{equation}
\label{eq:spectralgap}
\zeta_0 (\nu) = \tfrac{1}{2}\nu - \max \Big\{ \tfrac{1}{2} \boldsymbol{1}_{[2, \infty)}(\nu) \sqrt{\nu^2 - 4}\, ,\,  \tfrac{1}{2} \boldsymbol{1}_{[2 \sqrt{\Lambda_0}, \infty)}(\nu) \sqrt{\nu^2 - 4 \Lambda_0}\Big\}.
\end{equation}
\end{theorem}

The spectral gap \eqref{eq:spectralloc} in Theorem~\ref{mainspthm} determines the exponential decay for the evolutionary equation solution.
In \eqref{eq:spectralloc}, we regard the bound $\zeta_0(\nu)$ as uniform because $\cL$'s spectral bound $\Lambda_0$ is independent of $\nu$. 


\begin{remark}\label{remwhyH1}
The choice of the domain space $H^1\times L^2$ conveys a slight abuse of notation. Indeed, the operator $\cL$ in \eqref{eq:A} refers to its restriction to $H^1$, namely, the operator $\cL_{|H^1}:H^1\to L^2$ with domain $D(\widetilde{\cL}) = H^2$ given by
\[
\widetilde{\cL} u:= \cL u, \quad \text{ for every } \, u \in H^2,
\]
where, $\cL$ is the operator from $L^2$ to $L^2$ defined in \eqref{defL0}. However, since the original properties remain, such as closedness and spectral bounds, we keep the notation $\cL : H^1 \to L^2$ with the same dense domain $D(\cL) = H^2$ in the definition of the block matrix operator $\cA$. In the sequel, we shall remind the reader of this distinction at the steps of the proofs where it is explicitly required.
\end{remark}

\subsection{The perturbation equation} In order to establish the perturbation equations, we consider a solution $\brt(x) + u(x,t)$ to the dynamic equation \eqref{reddyneq}. Here, $u$ is the perturbation of the static N\'eel wall's phase, which, by the boundary conditions at infinity, must satisfy
\begin{equation}
 \label{bcu}
u(\pm \infty, t) = 0, \qquad t > 0.
\end{equation}
Upon substitution into \eqref{reddyneq}, we obtain the following nonlinear equation for the perturbation,
\begin{equation}
\label{nlpert}
\partial_t^2 u + \nu \partial_t u + \nabla \cE(\brt + u) = 0.
\end{equation}
In view of \eqref{defL0}, equation \eqref{nlpert} can be recast as
\[
 \partial_t^2 u + \nu \partial_t u + \cL u + \cN(u)  = 0,
\]
where $\cL u$ is the linearization around $\brt$ of $\nabla \mathcal{E}(\brt + u)$ acting on the perturbation $u$, and
\[
\cN(u) := \nabla \cE(\brt + u) - \cL u = O(u^2), 
\]
comprises the nonlinear terms. In view of the form of the operator \eqref{defL0}, we reckon the perturbation equation as a nonlinear wave equation. By the (standard) change of variables $v = \partial_t u$, the perturbation equation \eqref{nlpert} is equivalent to the the nonlinear hyperbolic system
\begin{equation}
\label{NLsyst}
\partial_t \begin{pmatrix}
            u \\ v
           \end{pmatrix} = \begin{pmatrix}0 & \Id \\ - \cL & - \nu \Id\end{pmatrix} \begin{pmatrix}u \\ v\end{pmatrix} + \begin{pmatrix}0 \\ \cN(u)\end{pmatrix},
\end{equation}
in the appropriate spaces, which will be determined later.

\subsection{The spectral problem}

By linearizing equation \eqref{nlpert} around the N\'eel wall's phase, we obtain the following equation for the perturbation,
\begin{equation}
 \label{linequ}
\partial_t^2 u + \nu \partial_t u + \cL u = 0,
\end{equation}
which is equivalent to the following linear system in the $(u,v)$ variables,
\begin{equation}
 \label{linsyst}
\partial_t \begin{pmatrix}
            u \\ v
           \end{pmatrix} = \begin{pmatrix}0 & \Id \\ - \cL & - \nu \Id\end{pmatrix} \begin{pmatrix}u \\ v\end{pmatrix}.
\end{equation}
We specialize the linearized equation \eqref{linequ} to perturbations of the form $e^{\lambda t}u(x)$, with $\lambda \in \C$ and $u \in X$, being $X$ a Banach space to be determined below. Substituting the Ansatz into \eqref{linequ}, we obtain the following spectral problem
\begin{equation}
 \label{spectralu}
(\lambda^2 + \nu \lambda) u + \cL u = 0.
\end{equation}

\begin{remark}
Under the substitution $\lambda = i \zeta$, equation \eqref{spectralu} can be written in terms of a \textit{quadratic operator pencil}, $\widetilde{\cT} u = 0$, with $\widetilde{\cT} = \widetilde{\cT}_0 + \zeta \widetilde{\cT}_1 + \zeta^2 \widetilde{\cT}_2$, and $\widetilde{\cT}_0 =\cL$, $\widetilde{\cT}_1 = i \nu \Id$, $\widetilde{\cT}_2 = -\Id$ (cf. Markus \cite{Ma88}). The transformation $v = \lambda u$ (the spectral equivalent of the change of variables $v = \partial_t u$)  defines an appropriate Cartesian product of the base space which allows us to write equation \eqref{spectralu} as a genuine eigenvalue problem of the form
\begin{equation}
 \label{evproblem}
  \cA \begin{pmatrix}
         u \\ v
        \end{pmatrix}:=
        \begin{pmatrix}
         0 & \Id \\ - \cL & -\nu
        \end{pmatrix} \begin{pmatrix}
         u \\ v
        \end{pmatrix} =
\lambda \begin{pmatrix}
         u \\ v
        \end{pmatrix} 
\end{equation}
The matrix operator $\cA$ is often called the companion matrix to the pencil $\widetilde{\cT}$ (see \cite{BrJoK14,KHKT13} for further information). Clearly, equation \eqref{evproblem} is the spectral equation associated to the linear system \eqref{linsyst}. We shall refer to both \eqref{spectralu} and \eqref{evproblem} as the spectral problem making no distinction.
\end{remark}

In the present stability analysis, we are interested in the spectral properties of the block matrix operator,
\[
 \cA : H^1 \times L^2 \to H^1 \times L^2, 
\]
regarded as a linear, densely defined operator in $H^1 \times L^2$ with domain $D(\cA) := H^2 \times H^1$. In other words, we choose our \emph{energy base space} as $H^1 \times L^2$. 
This choice of energy base space is not only consistent with the boundary conditions \eqref{bcu} for N\'eel wall's phase perturbations, but it also matches the appropriated perturbation space for the energy functional defined in \eqref{varprob}, which requires that variations of $u$ belong to $ H^1$. In addition, the condition  $v\in L^2$ implies that the perturbations have finite kinetic energy since $v$ is the spectral equivalent to $\partial_tu$. Thus, the stability analysis pertains to localized perturbations with finite energy in $H^1 \times L^2$. For shortness, we introduce the notation
\[
 U = (u,v) \in H^2 \times H^1, \qquad \cA U = (v, -\cL u - \nu v) \in H^1 \times L^2.
\]
The standard scalar product in $H^1 \times L^2$ will be denoted as
\[
\langle U, F \rangle_{H^1\times L^2}  = \langle u,f  \rangle_{H^1} + \langle v, g  \rangle_{L^2},
\]
for any $U = (u,v)$ and $F = (f,g)$ in  $H^1 \times L^2$. 

We finish this subsection by verifying the closedness of the operator $\cA$ so that the given resolvent and spectra definitions apply.

\begin{lemma}
\label{lemAclosed}
The block matrix operator $\cA : H^1 \times L^2 \to H^1 \times L^2$ is closed.
\end{lemma}
\begin{proof}
Assume that $\{U_j\}_{j \in \N} = \{(u_j, v_j)\}_{j \in \N} \subset D(\cA) = H^2 \times H^1$ is a Cauchy sequence in $H^1 \times L^2$ such that $\{ \cA U_j \}_{j \in \N}$ is a Cauchy sequence in $H^1 \times L^2$ as well.  Let us denote $U = (u,v) = \lim_{j \to \infty} U_j$ and $F = (f,g) = \lim_{j \to \infty} \cA U_j$. Hence, as $j \to \infty$, 
\[
\begin{aligned}
v_j &\to f, \quad \text{in } \, H^1,\\
-\cL u_j - \nu v_j &\to g, \quad \text{in } \, L^2.
\end{aligned}
\]
and in particular,  $v_j \to f$ in $L^2$ and $- \cL u_j \to g + \nu f$ in $L^2$. Now, because 
$\cL$ is closed 
(see Lemma \ref{thmLselfadj}), we have that $u_j \to u$ in $L^2$ with $u \in D(\cL) = H^2$ and $- \cL u = g + \nu f$. Therefore, $U = (u,v) \in D(\cA)$, 
\[
\cA U = (v, - \cL u - \nu v) = (f,g) = F,
\] 
and $\cA$ is a closed operator. 
\end{proof}

\subsection{Point spectral stability}
\label{subsec:pointstability}

In this section, we localize $\cA$'s point-spectral and study its eigenvalues associated with N\'eel wall's translational invariance.

\begin{lemma}
\label{lemzeroremains}
$\lambda = 0$ is a simple eigenvalue of $\cA$ with eigenfunction 
\begin{equation}
\label{defTheta}
\Theta := (\partial_x \brt, 0) \in D(\cA) = H^2 \times H^1.
\end{equation}
\end{lemma}

\begin{proof}
By Proposition \ref{propNeelw}, $\partial_x \brt \in H^2$, $\Theta \in D(\cA)$ and  $\cA \Theta = (0, - \cL \partial_x \brt) = 0$. Thus,
$0 \in \ptsp(\cA)$ with eigenfunction $\Theta$.   
The eigenvalue $\lambda=0$ has a geometric multiplicity equal to one. Otherwise, there exists a nontrivial $U = (u,v)$ in $\ker \cA\setminus \lspan \{\Theta\}$. Because
$\cA U=0$, we get $v = 0$ and $u = u_\perp + \alpha \partial_x \brt$ for some $\alpha \in \C$ where $\cL u_\perp = 0$. Now, by Lemma~\ref{corzeroL0}, $u_\perp = 0$ and $U \in \lspan \{\Theta\}$, a contradiction. 
 Furthermore, $\lambda=0$ algebraic multiplicity equals one. If this were not the case, there would exist a nontrivial Jordan chain given by $\cA U = \alpha \Theta$ where $\alpha\in \C\setminus \{0\}$ and $U\neq 0$; that is, $ \cA U = (v, - \cL u - \nu v) = (\alpha \partial_x \brt, 0)$. Hence, $v = \alpha \partial_x \brt$ and $-\cL u = \nu \alpha \partial_x \brt$ and $\opl$ has a nontrivial Jordan chain, which contradicts Lemma~\ref{corzeroL0}.   
\end{proof}

\begin{lemma}
\label{lemptspecstab}
Let $\lambda \in \ptsp(\cA)\setminus\{0\}$ and $\Lambda_0$ be the constant in Proposition \ref{propL0} \hyperref[propcc]{\rm{(c)}}.  Then,
\begin{equation} 
\label{ptspectbd}
\Re \lambda \leq - \tfrac{1}{2} \nu + \tfrac{1}{2} \boldsymbol{1}_{[2 \sqrt{\Lambda_0}, \infty)}(\nu) \sqrt{\nu^2 - 4 \Lambda_0}  < 0,
\end{equation}
where  $\boldsymbol{1}_\Omega(\cdot)$ denotes the characteristic function of the measurable set $\Omega \subset \R$.
\end{lemma}
\begin{proof}
Assume 
$\lambda \in \ptsp(\cA)\setminus\{0\}$, and  $U = (u,v) \in D(\cA) = H^2 \times H^1$ is such that $\cA U = \lambda U$. Thus, $\lambda u = v$,  $(\lambda + \nu)v + \cL u = 0$, and by substitution, we get 
\[
 \cL u + \lambda (\lambda + \nu) u = 0, 
\]
where $u \in H^2 = D(\cL)$. Therefore, $u$ is an eigenfunction for $-\lambda(\lambda + \nu) \in \ptsp(\cL)$, where $\lambda(\lambda + \nu) \in \R$ by $\cL$'s self-adjointness. By decomposing $L^2$ into $\lspan{\{\partial_x \brt\}}$ and its perpendicular space $L^2_\perp$, we get 
 $u = u_\perp + \alpha \partial_x \brt$, $v = v_\perp + \beta \partial_x \brt$, for some $\alpha, \beta \in \C$. We substitute the latter expressions for $u$ and $v$ into $\cA U = \lambda U$ to get 
 \[
\lambda u_\perp = v_\perp, \quad \beta = \lambda \alpha, \quad\text{ and }
\]
\[
\cL u_\perp + \lambda(\lambda + \nu) (u_\perp +  \alpha \partial_x \brt ) = 0.
\]
Because $\langle u_\perp, \partial_x \brt \rangle_{L^2} = 0$ and $\lambda(\lambda + \nu) \in \R$, by taking the $L^2$-product of $u_\perp$ with the latter equation,  we have that 
\[
\begin{aligned}
0 &= \langle \cL u_\perp, u_\perp \rangle_{L^2} + \lambda(\lambda + \nu) \| u_\perp \|_{L^2}^2 + \lambda(\lambda + \nu) \langle \alpha \partial_x \brt , u_\perp \rangle_{L^2} \geq (\Lambda_0 + \lambda^2 + \lambda \nu ) \| u_\perp \|_{L^2}^2.
\end{aligned}
\]
Hence, we obtain the bound 
\begin{equation}
\label{lambound}
\lambda ( \lambda + \nu) \leq - \Lambda_0,
\end{equation}
and the following relations for the real and imaginary parts of $\lambda(\nu+\lambda)$,  
\begin{subequations}
\label{las}
\begin{align}
\Im (\lambda(\lambda + \nu)) &= (\Im \lambda) (\nu + 2 \Re \lambda) = 0, \label{lasa}\\
- \Lambda_0 \geq \Re (\lambda(\lambda + \nu)) &= (\Re \lambda)^2 - (\Im \lambda)^2 + \nu \Re \lambda. \label{lasb}
\end{align}
\end{subequations}

By equation  \eqref{lasa}, either $\Re \lambda = - \tfrac{1}{2}\nu$
 or  $\lambda \in \R$.  The former case satisfies \eqref{ptspectbd}, so there is nothing to prove. Now, we assume $\lambda \in \R$ and consider two regimes for the physical parameter\footnote{Notice that $\cL$ and its spectral bound $\Lambda_0$ do not depend on $\nu$} $\nu>0$: (i) $\nu \in (0, 2 \sqrt{\Lambda_0})$ and (ii) $\nu \in [2 \sqrt{\Lambda_0},\infty)$. In case (i), equation \eqref{lasb} becomes $\lambda^2+\nu\lambda+\Lambda_0 \le 0$, but since its discriminant is negative, there are no real solutions to the inequality—a contradiction.  In case (ii), $\lambda^2+\nu\lambda+\Lambda_0 \le 0$ holds only for 
 \[
\lambda \in \big[ - \tfrac{1}{2}\nu - \tfrac{1}{2}\sqrt{\nu^2 - 4 \Lambda_0}, - \tfrac{1}{2}\nu + \tfrac{1}{2}\sqrt{\nu^2 - 4 \Lambda_0} \big].
\]
By combining all the above cases, we conclude the result.

\end{proof}


\subsection{Essential spectrum stability}
\label{subsec:essentialesp}
In this section, we study operator $\cA$'s essential spectrum and prove Theorem~\ref{mainspthm}. To that end, we define the following asymptotic block matrix operator,
\begin{equation}
\label{defAinf}
\cA_\infty : H^1 \times L^2 \to H^1 \times L^2, \qquad \cA_\infty := \begin{pmatrix} 0 & \Id \\ - \cL_\infty & - \nu \Id\end{pmatrix},
\end{equation}
with dense domain $D(\cA_\infty) = H^2 \times H^1$. Once again, with a slight abuse in notation, the operator $\cL_\infty$ in $\cA_\infty$ refers to  \eqref{defLinf}  restricted to  $H^1$, namely, 
\[
\widetilde{\cL}_\infty := {\cL_\infty}_{|H^1}, \quad \widetilde{\cL}_\infty : H^1 \to L^2,
\text{ and } D(\widetilde{\cL}_\infty) = H^2 \subset H^1, 
\]
that is $\widetilde{\cL}_\infty u := \cL_\infty u$ for every $u \in H^2$. In the sequel, we simplify notation by writing $\cL_\infty$ instead of its restriction $\widetilde{\cL}_\infty$.
Hence, $\cA_\infty$'s energy base space is $H^1 \times L^2$, and for any $U = (u,v) \in D(\cA_\infty)$, we have $\cA_\infty U = (v, - \cL_\infty u -\nu v) \in H^1 \times L^2$.

\begin{lemma}
\label{lemAinfcb}
The  operator $\cA_\infty : H^1 \times L^2 \to H^1 \times L^2$ is closed and onto.
\end{lemma}
\begin{proof}
The proof of the closedness of $\cA_\infty$ is the same as that of Lemma \ref{lemAclosed}, and we omit it. To show that $\cA_\infty$ is onto, notice that for any $F = (f,g) \in H^1 \times L^2$ the equation $\cA_\infty U = F$ with $U = (u,v) \in D(\cA_\infty) = H^2 \times H^1$ is equivalent to the system
\[
v = f, \qquad - \cL_\infty u = g + \nu f.
\]
By Lemma \ref{lemsolLinf}, there exists a unique solution $u \in H^2$ to the equation $- \cL_\infty u = g + \nu f \in L^2$. Since $v=f$ in $H^1$,   $ \cR(\cA_\infty)=H^1 \times L^2 $, as claimed.
\end{proof}

Thus, $\cA_\infty$ is a closed, densely defined operator with full range. The following result localizes its spectrum.

\begin{lemma}
\label{lemAinfsb}
Let $\lambda \in \sigma(\cA_\infty)$, then
\begin{equation}
\label{ddstar}
\Re \lambda \leq - \tfrac{1}{2} \nu + \tfrac{1}{2}  \boldsymbol{1}_{[2, \infty)}(\nu) \sqrt{\nu^2 - 4}  < 0.
\end{equation}
\end{lemma}
\begin{proof}

Let $\lambda \in \C$, $U = (u,v) \in H^2 \times H^1$ and $F = (f,g) \in H^1 \times L^2$ be such that $(\lambda-\cA_\infty) U = F$. Thus, 
$\lambda u - v = f$, $\cL_\infty u + (\lambda + \nu) v = g$. By substitution, we arrive at
\[
\big( \cL_\infty + \lambda(\lambda + \nu) \big) u = g + (\lambda + \nu) f,
\]
where $g + (\lambda + \nu) f \in L^2$ for any $\nu > 0$. By Remark \ref{lemressol}, there exist a unique solution $u \in H^2$ to the latter equation provided  $\lambda(\lambda + \nu) \in \C \backslash (-\infty, -1]$. Moreover, there exists 
$C( \lambda, \nu) > 0$, such that 
\[
\| u \|_{H^1} \leq \| u \|_{H^2} \leq C(\lambda,\nu) \| g + (\lambda + \nu) f \|_{L^2}.
\]
Now, because of the triangle inequality, we get
\[
\| U \|_{H^1 \times L^2}^2 = \| u \|_{H^1}^2 + \| v \|_{L^2}^2 = \| u \|_{H^1}^2 + \| f + \lambda u \|_{L^2}^2 \leq ( 1 + 2 |\lambda|) \| u \|_{H^1}^2 + 2 \|f \|_{L^2}^2,
\]
and we conclude that 
\[
\begin{aligned}
\| U \|_{H^1 \times L^2}^2 
\leq ( 1 + 2 |\lambda|) C(\lambda,\nu)^2 \| g + (\lambda + \nu) f \|_{L^2}^2 + 2 \| f \|_{L^2}^2
\leq 
\overline{C}(\lambda, \nu)  \| F \|_{H^1 \times L^2}^2,
\end{aligned}
\]
for some  $\overline{C}(\lambda, \nu)>0$. Therefore, $\lambda \in \rho(\cA_\infty)$, and 
$\sigma(\cA_\infty) \subset \big\{ \lambda \in \C \, : \, \lambda (\lambda + \nu) \in (-\infty, -1] \big\}$.
The relation that defines the latter set is equivalent to the following system of equations, 
\begin{equation}
\label{starry}
\begin{aligned}
\Im (\lambda(\lambda + \nu)) &= (\Im \lambda) (\nu + 2 \Re \lambda) = 0, \\
- 1 \geq \Re (\lambda(\lambda + \nu)) &= (\Re \lambda)^2 - (\Im \lambda)^2 + \nu \Re \lambda. 
\end{aligned}
\end{equation}
By \eqref{starry}'s first equation, either $\Re \lambda = - \tfrac{1}{2}\nu$
 or  $\lambda \in \R$.  The former case satisfies \eqref{ddstar}, so there is nothing to prove. Now, we assume $\lambda \in \R$ and consider two regimes for $\nu>0$: (i) $\nu \in (0, 2)$ and (ii) $\nu \in [2,\infty)$. In case (i), the \eqref{starry}'s second equation becomes $\lambda^2+\nu\lambda+1 \le 0$, but since its discriminant is negative, there are no real solutions to the inequality—a contradiction.  In case (ii), $\lambda^2+\nu\lambda+1 \le 0$ holds only for 
 \[
\lambda \in \big[ - \tfrac{1}{2}\nu - \tfrac{1}{2}\sqrt{\nu^2 - 4 }, - \tfrac{1}{2}\nu + \tfrac{1}{2}\sqrt{\nu^2 - 4 } \big].
\]
By combining all the above cases, we conclude the results.

\end{proof}

\begin{proof}[\bf Proof of Theorem \ref{mainspthm}]
By Weyl's spectrum splitting, $\sigma(\cA) = \ess(\cA)\cup \ptsp(\cA)$, and we study each part separately. 
Regarding $\cA$'s point spectrum, we have that 
\begin{equation}
\ptsp(\cA) \subset \{ 0 \} \cup \{ \lambda \in \C \, : \, \Re \lambda \leq - \tfrac{1}{2} \nu + \tfrac{1}{2} \sqrt{\nu^2 - 4 \Lambda_0} \, \boldsymbol{1}_{[2 \sqrt{\Lambda_0}, \infty)}(\nu) \},
\end{equation}
by lemmata \ref{lemzeroremains} and \ref{lemptspecstab}. 

For $\cA$'s essential spectrum, we claim first that operator $\cA$ is a relatively compact perturbation of $\cA_\infty$. Indeed,  let $\lambda \in \rho(\cA_\infty)$ and  $\{ U_j \}_{j \in \N}$ be a bounded sequence in $H^1 \times L^2$. Thus, $\{(\lambda - \cA_\infty)^{-1} U_j \}_{j\in\N}\subset D(\cA_\infty)$ is a bounded sequence in $H^2 \times H^1$ because  $(\lambda - \cA_\infty)^{-1}$ is a bounded operator. Let 
$( f_j, g_j) = ((\lambda - \cA_\infty)^{-1} U_j)^T$ where $U^T$ denotes vector's $U$ transpose. We have that 
\[
( \cA_\infty- \cA ) (\lambda - \cA_\infty)^{-1} U_j = \begin{pmatrix} 0 & 0 \\ \cL_\infty - \cL & 0\end{pmatrix} \begin{pmatrix} f_j \\ g_j 
\end{pmatrix} = \begin{pmatrix} 0 \\ (\cL_\infty- \cL) f_j
\end{pmatrix}.
\]
Now, because $\cL_\infty - \cL :H^2\to L^2$ is compact by Theorem \ref{main_res_L-Linf} and $\{ f_j \}_{j \in \N}$ is bounded in $H^2$, the sequence $\{(\cL_\infty - \cL) f_j\}_{j\in \N} \subset H^1$ is bounded and has a convergent subsequence in $L^2$. Therefore, $\{( \cA_\infty- \cA ) (\lambda - \cA_\infty)^{-1} U_j\}_{j\in \N}$ has a convergent subsequence in $H^1 \times L^2$, and  $( \cA_\infty- \cA ) (\lambda - \cA_\infty)^{-1} $ is compact operator on $H^1 \times L^2$ for every $\lambda \in \rho(\cA_\infty)$, as claimed. 

Now, since $\cA$ is a relatively compact perturbation of $\cA_\infty$,  $\ess (\cA) = \ess(\cA_\infty)$ by 
Weyl's essential spectrum theorem (see \cite{KaPro13}, p. 29). Hence, from Lemma \ref{lemAinfsb} we get  
\[
\ess (\cA) = \ess(\cA_\infty) \subset \{ \lambda \in \C \ | \ \Re \lambda \leq - \tfrac{1}{2} \nu + \tfrac{1}{2}  \boldsymbol{1}_{[2, \infty)}(\nu) \sqrt{\nu^2 - 4} \}.
\]

Finally, we combine $\ess(\cA)$ and $\ptsp(\cA)$'s bounds to conclude that 
$\sigma(\cA)$ satisfies \eqref{eq:spectralloc}. 
\end{proof}

\section{Semigroup generation and decay}
\label{secsemigroup}

In this section, we use semigroup theory to show that solutions to Cauchy problem \eqref{linsyst} generate a semigroup of quasi-contractions. Moreover, by restricting operator $\cA$ to a codimension-one subspace of $H^1\times L^2$, we obtain an exponentially decaying $C_0$-semigroup.


\begin{theorem}
\label{thm_semigroup_gen_1}
Let $\nu>0$ and $\cA$ as in \eqref{evproblem} with $D(\cA) = H^2\times H^1$ . Then, $\cA$ is the infinitesimal generator of a $C_0$-semigroup $\{e^{t\cA}\}_{t\geq 0}$ of quasicontractions, namely there exists $\omega \in \R$ such that 
\[
\|e^{t\cA}U\|_{H^1\times L^2} \leq e^{\omega t} \|U\|_{H^1\times L^2} ,
\]
for all $t \geq 0$ and every   $U\in H^1\times L^2$.
\end{theorem}

From basic semigroup theory (cf. \cite{EN00,Pa83}), the last theorem  implies 
\[
\frac{d}{dt} \big( e^{t\cA} U \big) = e^{t\cA} \cA U = \cA(e^{t\cA}U)
\]
for $U \in H^2 \times H^1$.

Before presenting our next result, we must define the following relevant spaces and sets. 
Let $\Phi_0:=(\nu\partial_x \brt,\partial_x \brt)$, then $H^1\times L^2=\lspan\{\Phi_0\}\oplus (H^1\times L^2)_{\perp}$ where $(H^1\times L^2)_{\perp}$ is  $\Phi_0$'s 
$L^2$-orthogonal complement, namely  
\begin{equation}
    \label{eq:setX_1}
    (H^1\times L^2)_{\perp}:=\{F\in H^1\times L^2 \ |\ \pld{F}{\Phi_0}=0\}.
\end{equation}
For $\nu>0$ fixed and $\cA$ is as in \eqref{evproblem} with $D(\cA) = H^2\times H^1$, let   
\begin{equation}
    \label{eq:setD_1}
D_\perp := \{U \in D(\cA)  \cap (H^1\times L^2)_{\perp} \ |\ \cA U \in (H^1\times L^2)_{\perp} \}.
\end{equation}

\begin{theorem}
\label{lemmaeight}
For $\cA$ as in \eqref{evproblem}, let $\cA_\perp : (H^1\times L^2)_{\perp} \to (H^1\times L^2)_{\perp}$ be its restriction to $D_\perp$ given by 
\begin{equation}
    \label{eq:opA_1}
 \cA_\perp U := \cA U, \qquad U \in D_\perp.
\end{equation}
Then, $\cA_\perp$ is the generator of an exponentially decaying $C_0$-semigroup $\{e^{t\cA_\perp}\}_{t\geq 0}$, namely,  there exists uniform constants $M \geq 1$ and $\omega_1 > 0$, such that
\begin{equation}
\label{lindecay}
\|e^{t \cA_\perp} U\|_{H^1\times L^2} \leq M e^{-\omega_1 t} \|U\|_{H^1\times L^2},
\end{equation}
for all $t \geq 0$ and every $U \in (H^1\times L^2)_{\perp}$.
%
\end{theorem}

\subsection{Generation of the semigroup}

In this subsection, we present the proof of Theorem~\ref{thm_semigroup_gen_1}. The  proof strategy is to apply the classical Lumer-Phillips theorem (see, e.g., Theorem 12.22, p. 407, in \cite{ReRo04}). To this end, we need to show that $D(\cA)$ is densely defined, $\cA-\nu \Id$ is onto, and a resolvent estimate on $\cA$, namely there exists $\eta_0>0$ such that $\Re \! \pwse{\cA U}{U}\le \eta_0\nwse{U}^2$. 
The remainder of this subsection is devoted to proving the latter three. 
\medskip

We begin by presenting some preparatory results.  The following result is necessary because the intersection does not distribute the direct sum.
  
\begin{lemma}\label{lem:H1split}
Let $L^2_\perp$ be  $\lspan\{\partial_x\brt\}$'s $L^2$-orthogonal complement, and $H^k_\perp:= H^k\cap L^2_\perp$ for $k\in\{1,2\}$. Then, for every $u\in H^k$ there exist unique $\alpha\in \C$ and $u_\perp\in H^k_\perp$ such that $u = u_\perp +\alpha\partial_x\brt$. Moreover, 
for every $U\in H^1\times L^2$ there exist unique $\alpha\in \C$ and $U_\perp\in H^k_\perp \times L^2$ such that $U = U_\perp + \alpha \Theta$ where $\Theta = (\partial_x\brt,0)$.
\end{lemma}

\begin{proof}
Let $k$ be fixed and $\bar{u}\in H^k$.
Because $L^2 = L^2_\perp  \oplus \lspan \{\partial_x\brt\}$, there exists unique $\alpha\in \C$ and $u_\perp\in L^2_\perp$ such that 
$u = u_\perp + \alpha \partial_x\brt$.
Moreover, since $ \partial_x\brt\in H^k$ by Proposition \ref{propNeelw} \hyperref[propc]{\rm{(c)}}, $u_\perp\in H^k$,
and the first follows. Now, the splitting for $U\in H^k\times L^2$ is a straight forward consequence of the $H^k$ case and the result follows.  
\end{proof}

\begin{definition}
\label{lem:sesqui_form}
The bi-linear form $a:H^1_\perp\times H^1_\perp\rightarrow \C$ associated to the elliptic operator $\cL$, namely $\pld{\opl u}{v} = a[u,v]$ for every $u\in H_\perp^2$ and $v\in H_\perp^1$,  is given by \begin{equation}
\label{eq:a}
a\left[ u, v\right] = \pld{\partial_x u}{\partial_x v} + b[s_{\theta}u,s_{\theta}v]-\pld{c_{\theta}u}{v}, 
\end{equation}    
where $b$ is as in \eqref{defbilinearB}. 
By definition, $a[\cdot,\cdot]$ is a sesquilinear Hermitian  form, and its positivity follows from Proposition \ref{propL0}.
\end{definition}

By the arguments in \cite{CMO07}, the inner products $a[\cdot,\cdot]$ and $\phu{\cdot}{\cdot}$ are equivalent in $H_\perp^1$. Denote by $\|u\|_a = \sqrt{a[u,u]}$ the induced norm, then there exist two constants $0< k_0 < K_0$ such that $k_0\|u\|_{H^1}\le \| u \|_{a} \le  K_0\|u\|_{H^1}$ for every $u\in H_\perp^1$.
In what follows, we also will consider the functional space $Z_\perp=H_\perp^1\times L^2$ with two norms,  $\| (u,v) \|_{Z_\perp} = \sqrt{\|u\|_a^2+\| v\|_{L^2}^2}$ and  $\|(u,v) \|_{2} =\|u\|_a +\| v\|_{L^2}$. 
By elementary algebra, $\sqrt{a^2+b^2} \leq a+b \leq \sqrt{2} \sqrt{a^2+b^2}$ for $a,b\geq 0$; therefore both norms are equivalent, that is   
\[\|U_\perp\|_{Z_\perp}\le \|U_\perp\|_{2} \le \sqrt{2}\|U_\perp \|_{Z_\perp}\] for every $U_\perp\in H_\perp^1$. 

\begin{lemma} \label{lem:equi_norms}
The space $Z_\perp=H_\perp^1\times L^2$ is a Hilbert space with respect to the inner product $\pwse{\cdot}{\cdot}$, and there exist constants $0<k < 1 < K$ such that 
\[
k\|U_\perp\|_{H^1\times L^2} \le \| U_\perp \|_{Z_\perp} \le K \|U_\perp\|_{H^1\times L^2} 
\]
for every $U_\perp\in Z_\perp$, where $\| \cdot\|_{Z_\perp}$ is as in the latter paragraph.  Moreover, $\|\cdot\|_{Z_\perp}$ is induced by the inner product $\langle\cdot,\cdot\rangle_{Z_\perp}:Z_\perp \times Z_\perp\to \C$ given by
    \[
    \langle U_\perp,V_\perp\rangle_{Z_\perp} := a[u,w] + \pld{v}{z},  \]
    for $U_\perp = (u,v)$ and $V_\perp = (w,z)$ in $Z_\perp$.  Henceforth, $\langle\cdot,\cdot\rangle_{Z_\perp}$ and  $\pws{\cdot}{\cdot}$ are equivalent in $Z_\perp$. 
\end{lemma}

Lemma~\ref{lem:equi_norms} is relevant because the constants on the norm equivalence satisfy $k < 1 < K$.

\begin{proof}[Proof of Lemma~\ref{lem:equi_norms}]
The Cartesian product $H_\perp^1\times L^2$ is also a Hilbert space with respect to the induced inner product $\|\cdot,\cdot\|_{Z_\perp}$ since $a[\cdot,\cdot]$ and $\phu{\cdot}{\cdot}$ are equivalent in $H^1_\perp$ which is a Hilbert space equipped with the latter inner product. Now, because $ k_0\|u\|_{H^1}\le \|u\|_{a} \le K_0\|u \|_{H^1}$ for every  $u\in H_\perp^1$, and $ \|U_\perp\|_{Z_\perp}\le \|U_\perp\|_{2} \le \sqrt{2}\|U_\perp \|_{Z_\perp}$ for every $U_\perp\in H_\perp^1$, we have that 
\[
\min\{1/\sqrt{2},k_0/\sqrt{2}\}\|U_\perp\|_{Z_\perp}\le \|U_\perp\|_{H^1\times L^2} \le \max\{1,K_0\}\|U_\perp \|_{Z_\perp}.
\]
The rest of the lemma's claims follow directly from the latter equivalence between the norms. 
\end{proof}

\begin{lemma}
\label{lem:prodZ}
For $U,V\in Z=H^1\times L^2$, let $U=U_\perp+ \alpha\Theta$ and $V=V_\perp+\beta\Theta$ be the decomposition of Lemma~\ref{lem:H1split}, and $\langle \cdot,
\cdot\rangle_{Z_\perp}$ as in Lemma~\ref{lem:equi_norms}. 
We define  
\[
\langle U, V\rangle_{Z} 
:=\langle U_\perp,V_\perp\rangle_{Z_\perp} + 
\pwse{U}{\beta \Theta} + \pwse{\alpha\Theta}{V} + \alpha\beta^* 
\nwse{\Theta}^2.
\]
Then, $\langle\cdot,\cdot\rangle_{ Z}: Z\times Z \to \C$ is an inner product equivalent to $\pwse{\cdot}{\cdot}$.
\end{lemma}

\begin{proof}
The form $\langle\cdot,\cdot\rangle_{ Z}: Z\times Z \to \C$ is Hermitian and sesquilinear because it is the sum of four sesquiliner Hermitian inner products in the corresponding functional spaces.  Next, for $U=U_\perp+ \alpha\Theta\in Z$, we have that
\[
\langle U,U\rangle_{Z} =\|U_\perp\|_{Z_\perp}^2 + 2 \Re\!\!\pwse{U}{\alpha\Theta}+\nwse{\alpha\Theta}^2.
\]
Adding and subtracting $\nwse{U_\perp}^2$ to the latter equation, yields
\begin{equation}
\label{eq:posdefX}
 \langle U,U\rangle_{Z} =\|U_\perp\|_{Z_\perp}^2 + \nwse{U}^2-\nwse{U_\perp}^2.
\end{equation}
Now, by Lemma~\ref{lem:equi_norms}, there exists positive constants $k<1<K$ such that 
\[
k\|U_\perp\|_{H^1\times L^2} \le \| U_\perp \|_{Z_\perp} \le K \|U_\perp\|_{H^1\times L^2}. 
\]
Combining this equivalence between norms with \eqref{eq:posdefX}, we get 
 \[
 \nwse{U}^2-(1-k^2)\nwse{U_\perp}^2
 \le \langle U,U\rangle_{Z} 
\le  \nwse{U}^2+(K^2-1)\nwse{U_\perp}^2, 
 \]
and because $\nwse{U}^2\ge \nwse{U_\perp}^2$ with equality only if $\alpha=0$, we conclude
\[
 k^2\nwse{U}^2
 \le \langle U,U\rangle_{Z} 
\le  K^2\nwse{U}^2.
 \]
Therefore, $\langle\cdot,\cdot\rangle_{ Z}$ and $\pwse{\cdot}{\cdot}$ are equivalent inner products in $Z$, as claimed. 
\end{proof}

\begin{lemma} \label{lemF4} There exists $\eta_0>0$ depending only on $\brt$,  such that 
\[
\Re \! \pwse{\cA  U}{ U}\leq \eta_0 \nwse{U}^2,
\]
for every $ U \in D(\cA)$.
\end{lemma}

\begin{proof}
Recall that  $D(\cA)=H^2\times H^1$, denote $Z_\perp = H^1_\perp\times L^2$, and let $U = U_\perp + \alpha \Theta\in D(\cA)$ as in Lemma~\ref{lem:H1split}. Regarding $U_\perp$'s components, $U_\perp=(u_\perp,v)$ where $u_\perp\in H^1_\perp$ and $v=v_\perp + \beta \partial_x\brt\in H^1$ with $v_\perp\in H^1_\perp$ and $\beta\in \C$.

By Lemma~\ref{lemzeroremains},
$\Theta$ is an eigenfunction with zero eigenvalue $\cA$. Thus, 
\[
\cA  U = \cA U_\perp = V_\perp+\beta \Theta, \quad \text{where } \quad V_\perp=(
    v_\perp, -\nu v-\opl u_\perp)\in Z_\perp, 
\]
and 
\begin{equation}
\label{eq:auxLP0}
    \left\langle \cA  U, U\right\rangle_{Z}
=\langle V_\perp,U_\perp\rangle_{Z_\perp} + 
     \pwse{V}{\alpha\Theta} + \pwse{\beta \Theta}{U} + \alpha^*\beta 
    \nwse{\Theta}^2,
\end{equation}
for $\left\langle \cdot, \cdot\right\rangle_{Z}$   as given in Lemma~\ref{lem:prodZ} and $V=  V_\perp+\beta \Theta$. 
By Lemma~\ref{lem:equi_norms},
we have that
\[
\langle V_\perp,U_\perp\rangle_{Z_\perp} 
= a[v_\perp,u_\perp] 
-\langle \opl u_\perp,v_\perp\rangle_{L^2} 
-\nu\langle  v,v\rangle_{L^2} 
-\beta^*\langle \opl u_\perp,
\partial_x\brt
\rangle,
\]
and since $\langle \opl u_\perp,v_\perp\rangle_{L^2} =a[u_\perp,v_\perp]$, $\opl$ is self adjoint, and $\opl\partial_x\brt=0$ by Proposition~\ref{propL0},  
we get
\begin{equation}
\label{eq:auxLP1}
\langle V_\perp,U_\perp\rangle_{Z_\perp} 
= 2i \, \Im a[w,u] 
-\nu\|v\|_{L^2}^2.  
\end{equation}
Using integration by parts, and that 
$\Theta=(\partial_x\brt,0)$, we find
\begin{equation}
\label{eq:auxLP2}
\langle V, \alpha\Theta\rangle_{H^1\times L^2} = \langle \partial_x v_\perp, \alpha\partial_x^2\brt\rangle_{L^2}
=-\langle  v_\perp, \alpha\partial_x^3\brt\rangle_{L^2}.
\end{equation}
Thus, by substitution of \eqref{eq:auxLP1} and \eqref{eq:auxLP2} into \eqref{eq:auxLP0}, we get
\[
 \left\langle \cA  U, U\right\rangle_{Z}
 = 2i \, \Im a[w,u] 
-\nu\|v\|_{L^2}^2 
-\langle  v_\perp,\alpha\partial_x^3\brt\rangle_{L^2}
+\pwse{\beta \Theta}{U} + \alpha^*\beta 
    \nwse{\Theta}^2,
\]
and, using Cauchy-Schwarz inequality, the real part of the latter equation is bounded as 
\[
2\Re \left\langle \cA  U, U\right\rangle_{Z}
\le 
\|v_\perp\|_{L^2}^2
+ \nwse{U_\perp}^2 
+ |\alpha|^2\|\partial_x^3\brt\|_{L^2}^2
+ ( |\beta|^2+2\alpha^*\beta) 
    \nwse{\Theta}^2.
\]
Since $\nld{\partial^3_{x}\brt}<\infty$ and $\nld{\partial_x\brt}\neq 0$ by  Proposition~\ref{propNeelw}, we define the constants
$C_1 :=\nws{\Theta}^2/\nld{\partial_x \brt}^2$ and $C_2:=\nld{\partial^3_x \brt}^2/\nld{\partial_x \brt}^2$, to get
\[
2\Re \left\langle \cA  U, U\right\rangle_{Z}
\le 
\|v_\perp\|_{L^2}^2
+ \nwse{U_\perp}^2 
+ (|\alpha|^2C_1 
+(|\beta|^2+2\alpha^*\beta) C_2)
\|\partial_x\brt\|_{L^2}^2
\]
Because $\max\{\|v_\perp\|_{L^2},  
|\beta|^2\|\partial_x\brt\|_{L^2} \} \le \|v\|_{L^2}^2 \le \nwse{U}^2  $, 
$ \nwse{U_\perp}^2 \le \nwse{U}^2 $, we find
\[
2\Re \left\langle \cA  U, U\right\rangle_{Z}
\le (3 + 2C_2 + 2C_1)
\nwse{U}^2.
\]
Finally, the result follows because $\langle\cdot,\cdot\rangle_{ Z}$ and  $\pwse{\cdot}{\cdot}$ are equivalent inner products by Lemma~\ref{lem:prodZ}, and the resulting constant $\eta_0=K^2(3/2+C_2+C_1)>0$ only depends on $\brt$.   
\end{proof}

\begin{proof}[\bf Proof of Theorem~\ref{thm_semigroup_gen_1}]
From Lemma~\ref{lemF4}, there exists $\eta_0>0$ such that 
\[
\Re \! \pwse{\cA  U}{ U}\leq \eta_0 \nwse{U}^2,
\]
for every $ U \in D(\cA)$. In addition, $\R^+\subset \rho(\cA)$ by Theorem~\ref{mainspthm}. Thus, by choosing $\tau > \eta_0$ with $\eta_0$ as above, we conclude that $\cA-\tau$ is onto.  Finally, 
 $\cA$ is densely defined in $H^1\times L^2$ by definition. 
 Therefore, the result follows by the classical Lumer-Phillips theorem.
\end{proof}

\subsection{The adjoint operator}
Our proof of the semigroup's exponentially decay requires the analysis of $\cA$'s formal adjoint operator $\cA^*$. This subsection is devoted to presenting this analysis. 

In the present context, because $H^1$ and $L^2$ are reflexive Hilbert spaces, then $\cA : H^1\times L^2 \to H^1 \times L^2$ with $D(\cA) = H^2 \times H^1$ has a formal adjoint which is also densely defined and closed. Moreover, $\cA^{**} = \cA$ (cf. \cite{Kat80}, Theorem 5.29, p. 168). 

\begin{lemma} 
 \label{Thetanotzero}
    The formal adjoint $\cA^*$, restricted to the domain $D(\cA)$, is given by
    \begin{equation}
    \label{eqA*}
    \left.\cA^*\right|_{D(\cA)} = \begin{pmatrix}
        0 & \cF \\ -\partial_{xx}+\Id & -\nu
    \end{pmatrix}
    \end{equation}
    where the operator $\cF:H^1\to H^{-1}$ is formally defined as the map 
    \[
   v\mapsto \cF v = -(\cS v-c_{\theta}v,\partial_x v) .
   \] Moreover, $\cF|_{H^2}=[1+(-\Delta)]^{-1}\opl$, where $[1+(-\Delta)]^{-1} u$ denotes the convolution of $u$ with the Bessel potential of order $2$.
\end{lemma} 

\begin{proof}
    Denote $Z=H^1\times L^2$, and 
    let $U=(u,v)$ and $V=(w,z)$ be both in $D(\cA) = H^2 \times H^1$. By definition of the inner product in $Z$, we have
    \[
    \begin{aligned}
    \pwse{\cA U}{V} =  \phu{v}{w} - \pld{\opl u + \nu v}{z} = & \pld{v}{w-\nu z} -\pld{\opl u}{z}+\pld{\partial_x v}{\partial_x w}.
    \end{aligned}
    \]
    Because of $w\in H^2$, integration by parts on $\pld{\partial_x v}{\partial_x w}$ yields
    \begin{equation}
\label{eq:innerprodws}
    \pwse{\cA U}{V} = \pld{v}{-\partial^2_x w+w-\nu z} -\pld{\opl u}{z},
    \end{equation}
and by the symmetry of $\cS$ (see Lemma \ref{cor:operatorS}), we get
\[
    \pwse{\cA U}{V}= 
\pld{v}{-\partial^2_x w + w-\nu z} -\pld{\partial_xu}{\partial_x z} -\pld{u}{\cS z -c_{\theta}z}.
 \]
Therefore, $\pwse{\cA U}{V} = \pwse{U}{\cA^*V}$ for $\cA^*$ as in \eqref{eqA*} where $\cF z =  -(\cS z-c_{\theta}z,\partial_x z)\in H^{-1}$. 

The Bessel potential's Fourier symbol of order $k=2$ is $\hat{\cK}=(1+|\xi|^2)^{-1}$. Since $\opl$ is self adjoint, by Plancherel's identity applied twice to \eqref{eq:innerprodws}'s last term, we get 
 \[
    \pld{\opl u}{z}=\pld{u}{\opl z} = \pld{\hat{u}}{\widehat{\opl z}(\xi)} = 
\pld{\hat{u}}{(1+|\xi|^2)\widehat{\cK \opl z}(\xi)}=\phu{u}{\cK\opl z},
    \]
where $z\in H^2$ and last equality holds because $\cK \opl z\in H^1$ with $\nhu{\cK \opl z} ^2\leq\nld{\opl z}^2$.

    \end{proof}

\begin{corollary}
    Let $\cA^*$ be the formal adjoint of $\cA$. Also, let $\left.\cA^*\right|_{D(\cA)}$ and $\cF$ be as in Lemma \ref{Thetanotzero} and define 
    \begin{equation}
    \label{eq:phi}
    \Phi := (
        \nu[1+(-\Delta)]^{-1}\ \partial_x \brt, \partial_x \brt).
        \end{equation}
    Then $\Phi \in X^*$ is an eigenvector of the adjoint $\cA^*:X^*\to X^*$, associated to the isolated, simple eigenvalue $\lambda=0$.
\end{corollary}
\begin{proof}
By Lemma~\ref{lemzeroremains}, the operator $\cA:H^1\times L^2 \to H^1\times L^2$ has $\lambda=0$ as a simple eigenvalue, therefore $\lambda=0$ is also eigenvalue of $\cA^*$ with the same geometric and algebraic multiplicities (see Kato \cite{Kat80}, Remark 6.23, p. 184). 

By the Plancherel's identity we have 
\begin{equation}
\label{eq:firstentry}
\nhd{[1+(-\Delta)]^{-1}\ \partial_x \brt}^2 = \nu^2 \int_\R(1+|\xi|^2)^2(1+|\xi|^2)^{-2}\left|\widehat{\partial_x\brt}\right|^2d\xi = \nu^2\nld{\partial_x\brt}^2.
\end{equation}
Thus, $[1+(-\Delta)]^{-1}\ \partial_x \brt\in H^2$, and  
$\Phi\in H^2\times H^2\subset D(\cA)$ by \hyperref[propc]{\rm{(c)}} in Proposition \ref{propNeelw}. Now, because $H^2\subset H^{-1}$ and $(H^1\times L^2)^* = H^{-1}\times L^2$ in the norm $\nwse{\cdot}$, we have that $\Phi\in (H^1\times L^2)^*$. 
Finally, from Lemma \ref{Thetanotzero}, 
\[
    \cA^*\Phi = \left.\cA^*\right|_{D(\cA)}\Phi=( \cF \partial_x\brt ,0) =( \cK\opl \partial_x\brt ,0)= (0,0),
\]
where last equality holds due to Bessel potential's $L^2$ invertibility, and that $\opl \partial_x\brt =0$.
\end{proof}

The following result gives the explicit representation of  the functional $\pwse{\cdot}{\Phi}\in (H^1\times L^2)^*$ in the dual of $L^2\times L^2$, and the projection operator  $\cP: H^1\times L^2 \to (H^1\times L^2)_\perp$. 

\begin{lemma}
\label{char_work_space}
Let $\Phi$ as in \eqref{eq:phi},  $\Phi_0:=(\nu \partial_x\brt, \partial_x\brt )\in L^2\times L^2$ and $\Theta = (\partial_x\brt,0)$. Then, 
\begin{equation}
\label{eq:chractPhi}
\pwse{F}{\Phi} =\pld{F}{\Phi_0},
\end{equation}
 for every $F \in H^1\times L^2$.   Moreover, the operator $\cP: H^1\times L^2 \to (H^1\times L^2)_\perp$ given by 
\[\cP U := U - \pwse{\Theta}{\Phi}^{-1}\pwse{U}{\Phi}\Theta,\]
is a projector with 
\[
\cR(\cP) =\left\{F\in H^1\times L^2 \ \middle|\ \pld{F}{\Phi_0} =0 \right\}.\]
\end{lemma}
\begin{proof}
    We argue in Fourier space; let $F=(f,g)  \in H^1\times L^2$, then
    \[
    \pwse{F}{\Phi} = \phu{f}{\nu[1+(-\Delta)]^{-1}\partial_x \brt} + \pld{g}{\partial_x \brt}= \pld{f}{\nu \partial_x\brt} +\pld{g}{\partial_x \brt}.
    \]
    Hence, $\pwse{F}{\Phi} = \pld{F}{\Phi_0}$, as claimed.     
Next, we note that 
$\pwse{\Theta}{\Phi}=\pld{\Theta}{\Phi_0} = \nu \nld{\partial_x\brt}^2 > 0$ so that  $\cP$ is well defined. Since, $\cP$ is obviously linear and $\cP F = F$ for every $F\in (H^1\times L^2)_\perp$, it is a projection. The $\cR(\cP)$ characterization is a consequence of \eqref{eq:chractPhi}.
\end{proof}

\subsection{Exponential decay of the semigroup}
In this subsection, we show that the 
the restricted operator $\cA_\perp$ is the generator of a exponentially decaying $C_0$-semigroup, namely we prove Theorem~\ref{lemmaeight}. The proof strategy is to apply Gearhart-Pr\"uss theorem. To this end, 
we need to show that 
\begin{equation}\label{eq:resolventest}
\sup_{\Re \lambda>0} \|(\lambda-\cA_\perp)^{-1}\|_{(H^1\times L^2)_\perp \to (H^1\times L^2)_\perp}<\infty.
\end{equation}
Condition \eqref{eq:resolventest} is satisfied provided that any solution $U\in (H^1\times L^2)_\perp$ to the linear equation $(\lambda-\cA_1)U =F$ for $F\in H^1\times L^2$ satisfies a resolvent estimate of the form $\nwse{U}\leq C(\lambda)\nwse{F}$, in which the constant $C(\lambda)$ is bounded for every $\Re \lambda>0$.

We begin with some preparatory results.

\begin{lemma}
\label{lemma:Pcommutes}
Let $\cP:H^1\times L^2 \to (H^1\times L^2)_\perp$ be the projection operator as in Lemma~\ref{char_work_space}. Then,  
$e^{t \cA}  \cP = \cP e^{t \cA}$ for  all $t \geq 0$.
\end{lemma}

\begin{proof}
In the reflexive Banach space $H^1\times L^2$, weak and weak$^*$ topologies coincide. Therefore, the family of dual operators $\{ (e^{t \cA})^*\}_{t\geq 0}$, consisting of all the corresponding formal adjoints in $(H^1\times L^2)^*$, is also $C_0$-semigroup (cf. \cite{EN00}, p. 44), and the infinitesimal generator of this semigroup is $\cA^*$ (see Corollary 10.6 in \cite{Pa83}). Hence, 
\begin{equation}
\label{eq:deulexp}
    (e^{t \cA})^* = e^{t \cA^*}.
\end{equation}
Next, let  $U\in H^1\times L^2$, we have 
\[
\cP e^{t \cA} U = e^{t \cA} U-\pwse{\Theta}{\Phi}^{-1}\pwse{e^{t \cA}U}{\Phi}\Theta. 
\]
Because of \eqref{eq:deulexp}, we note that
\[
\pwse{e^{t \cA}U}{\Phi}\Theta 
= \pwse{U}{e^{t \cA^*}\Phi}\Theta
=\pwse{U}{\Phi}e^{t \cA} \Theta,
\]
where we used that $e^{t \cA^*} \Phi = \Phi$ and $e^{t \cA} \Theta = \Theta$ for the last equality.
Therefore, combining the above equations we find $\cP e^{t \cA} U =
e^{t \cA} \cP U$, as claimed. 
\end{proof}

\begin{lemma} 
\label{elcoro}
Let $(H^1\times L^2)_\perp$ and  $\cD_\perp$ be as in \eqref{eq:setX_1} and \eqref{eq:setD_1}, respectively.
Assume $\cA_\perp$ is as in Theorem~\ref{lemmaeight},
then
\[
 \sigma (\cA_\perp) \subset \{\lambda \in \C \, : \, \Re  \lambda \leq - \zeta_0(\nu) < 0\},
\]
with $\zeta_0(\nu)$ given by \eqref{eq:spectralgap}.
 In particular, $\lambda = 0$ is not spectrum of $\cA_\perp$.  
\end{lemma}

\begin{proof}
By the spectral decomposition theorem  $\sigma(\cA_\perp)\subset \sigma(\cA)$ (see \cite{Kat80}), but $0\notin \sigma(\cA_\perp)$ because $\cP \Theta = 0$ and $\Theta\neq 0$  so that $\Theta \notin (H^1\times L^2)_\perp$. Therefore, by Theorem \ref{mainspthm}, we have that $\sigma (\cA_\perp) \subset \sigma (\cA_\perp)\setminus\{0\} \subset \{\lambda \in \C \, : \, \Re  \lambda \leq - \zeta_0(\nu) < 0\}$ as claimed.
\end{proof}

\begin{lemma}
 \label{lemmaseven}
The family of operators $\{e^{t \cA_\perp}\}_{t\geq 0}$, defined as 
\[
 e^{t \cA_\perp}U := e^{t \cA} U, 
\]
for $U \in (H^1\times L^2)_\perp$ and  $t \geq 0$, is a $C_0$-semigroup of quasicontractions in the Hilbert space $(H^1\times L^2)_\perp$ with infinitesimal generator $\cA_\perp$.
\end{lemma}
\begin{proof}
From Lemma \ref{lemma:Pcommutes} we concludes that $(H^1\times L^2)_\perp$  is an $e^{t \cA}$-invariant closed Hilbert subspace of $H^1 \times L^2$. Therefore, $\cA_\perp$ is the restricted semigroup infinitesimal generator of $\{e^{t \cA_\perp}\}_{t\geq 0}$ (see Section 2.3 of \cite{EN00}), where  the semigroup properties are inherited from those of $\{e^{t \cA}\}_{t\ge 0}$ in $ H^1 \times L^2$. 
\end{proof}

In the next result, we show the resolvent estimate needed to show \eqref{eq:resolventest}. 

\begin{lemma} \label{lem:main_ineq}
Let $C_0$ and $C_1$ be two positive constants, and  $k$ and $K$ the constants in Lemma~\ref{lem:equi_norms}. For $\lambda\in \res{\cA}$, assume that $f,g,u,v\in L^2_\perp$  are such that $F_\perp=(f,g) \in (H^1\times L^2)_\perp$, $U_\perp = (u,v) \in D_\perp$ and 
\[
(\lambda-\cA_\perp)U_\perp =F_\perp.
\]
Then, if $\Re\lambda > C_0$ or $|\Im\lambda | > C_1$, the following estimate holds
\[
\nwse{U_\perp}\leq C(C_0,C_1,\nu) \nwse{F_\perp},
\]
where  
\[
C(C_0,C_1,\nu) = \frac{2\sqrt{2} K}{k C_0} \quad\text{or}\quad
\frac{2K\sqrt{8C_1^2+2\nu^2}}{k\nu C_1},
 \]
whenever $\Re\lambda > C_0$ or $|\Im\lambda | > C_1$, respectively. 
\end{lemma}



\begin{proof}
First, we claim that, under the lemma's assumptions,
\begin{equation}
\label{eq:used_ineq}
\left|\lambda^* a[u,u]+(\lambda+\nu)\nld{v}^2\right| \leq \|U_\perp\|_2\|F_\perp\|_2.
\end{equation}

Indeed, in terms of its components $(\lambda-\cA_\perp)U_\perp =F_\perp$ is given by 
\[
\lambda u - v = f
\quad\text{and}\quad
\opl u +(\lambda+\nu)v = g.
\]
Take the $L^2$-product of latter equations, with $\opl u$ and $v$, respectively. The result is
\[
\lambda^*\pld{\opl u}{u} -\pld{\opl u}{v} = \pld{\opl u}{f} 
\quad\text{and}\quad
\pld{\opl u}{v} +(\lambda+\nu)\nld{v}^2 = \pld{g}{v} .
\]
 By Definition~\ref{lem:sesqui_form}, we write the latter equations in terms of the sesquilinear form $a[\cdot,\cdot]$ as  
\[
\lambda^* a[u,u] - a[u,v] = a[u,f], \quad\text{and}\quad a[u,v] +(\lambda+\nu)\nld{v}^2 = \pld{g}{v}.\]
Then, adding both equations, taking its complex modulus, and using triangle inequality yields
\[
\left | \lambda^* a[u,u]+(\lambda+\nu)\nld{v}^2\right | \leq | a[u,f]|+ |\pld{g}{v}|.
\]
Because $a$ is a Hermitian sesquilinear nonnegative form, Cauchy-Schwarz inequality holds for $a$ in $H^1_\perp$. Hence, 
\[
\left | \lambda^* a[u,u]+(\lambda+\nu)\nld{v}^2\right | \leq  a^{1/2}[u,u]a^{1/2}[f,f]+ \nld{g}\nld{v},
\]
and inequality \eqref{eq:used_ineq} follows since the right hand side is bounded by
\[
\left[\nld{v}+a^{1/2}[u,u]\right]\left[\nld{g}+a^{1/2}[f,f]\right]= \|U\|_2\|F\|_2. 
\]

Second, let $a,b,\nu\in \R^+$,  and $\lambda\in \C$. 
By elementary algebra,
\begin{equation}
\label{eq:complex}
 |\lambda^* a + (\lambda+\nu)b|^2 = (\Re\lambda \ a + (\Re\lambda+\nu) b)^2 + (\Im\lambda)^2 (a-b)^2.  
\end{equation}
Thus, assuming $\Re \lambda >0$, we get 
$|\lambda^* a + (\lambda+\nu)b|^2 \ge  (\Re\lambda)^2 (a + b)^2$.
By letting $a=a[u,u]$ and $b=\nld{v}^2$ in the last inequality, we find
\[
\left | \lambda^* a[u,u]+(\lambda+\nu)\nld{v}^2\right |
\ge \Re\lambda \|U_\perp\|_{Z_\perp}^2.
\]
 Now, by the equivalence between the the norms $\|\cdot\|_{Z_\perp}$ and $\|\cdot\|_{2}$ (see paragraph before Lemma~\ref{lem:equi_norms}), we get
\begin{equation}\label{eq:resolvent_aux2}
\left | \lambda^* a[u,u]+(\lambda+\nu)\nld{v}^2\right |
\ge \frac{\Re\lambda}{2} \|U_\perp\|_{2}^2.
\end{equation}
Combining \eqref{eq:used_ineq} and  \eqref{eq:resolvent_aux2}, we obtain
\begin{equation}
    \label{eq:des2Re}
  \|U_\perp\|_{2}\le \frac{2}{\Re\lambda} \|F_\perp\|_{2} \le K(C_0)\|F_\perp\|_{2}  
\end{equation}
whenever $\Re\lambda > C_0$ and $K(C_0)= 2/C_0$. 

Now, from \eqref{eq:complex} we also get    
\[
 |\lambda^* a + (\lambda+\nu)b|^2 \ge 
 \nu^2 b^2 + (\Im\lambda)^2 (a-b)^2.
\]
Let $a=r^2\sin^2 t$ and $b=r^2\cos^2t$ for $t\in [0,\pi/2]$. Thus, after this change of variables and some computations, we get 
\[
 |\lambda^* a + (\lambda+\nu)b| \ge 
\left[
 \left((\Im\lambda)^2 + \frac{\nu^2}{4}\right) \cos^2(2t) + \frac{\nu^2}{2}\cos(2t) + \frac{\nu^2}{4}
 \right]^\frac{1}{2}r^2
  \ge 
\left[
\frac{\nu^2(\Im\lambda)^2}{4(\Im\lambda)^2+\nu^2}\right]^\frac{1}{2}  r^2,
\]
where the last inequality follows by computing the quadratic equation minimum for $\cos(2t)$, which is attained at  $\cos(2t)= -\nu^2/(4\Im(\lambda)^2+4\nu^2)$. 
Therefore, by letting $a=a[u,u]$ and $b=\nld{v}^2$, we have that $r^2=a[u,u]+\nld{v}^2= \|U_\perp\|_{Z_\perp}^2$. In addition, from the last inequality, we get
\[
 |\lambda^* a[u,u] + (\lambda+\nu)\nld{v}^2| \ge 
\frac{\nu |\Im\lambda|}{\sqrt{4(\Im\lambda)^2+\nu^2}}
 \|U_\perp\|_{Z_\perp}^2.
\]
Arguing as before, we conclude that
\begin{equation}
\label{eq:des2Im} 
\|U_\perp\|_{2} \le K(C_1,\nu)\|F_\perp\|_{2}, 
\end{equation}
whenever $|\Im\lambda|\ge C_1$ and  
$ K(C_1,\nu) = 2
\sqrt{4C_1^2+\nu^2}/\nu C_1$.
Finally, the results follows from \eqref{eq:des2Re}, \eqref{eq:des2Im} and the equivalence between norms in Lemma~\ref{lem:equi_norms}.
\end{proof}

\begin{proof}[\bf Proof of Theorem \ref{lemmaeight}] 
By Lemma \ref{lemmaseven}, the operator $\cA_\perp: (H^1\times L^2)_\perp\to (H^1\times L^2)_\perp$ is the infinitesimal generator of $C_0$-semigroup $\{e^{t \cA_1}\}_{t\geq 0}$ of quasicontractions. 
By Lemma~\ref{elcoro},  $\{\lambda \in \C \, : \, \Re  \lambda \leq - \zeta_0(\nu) < 0\}\subset \rho(A_\perp)$.  Therefore, the result follows from the Gearhart-Prüss theorem provided we show that 
\begin{equation}
    \label{lemF6}
 \sup_{\Re \lambda > 0} \| (\lambda-\cA_\perp)^{-1} \|_{(H^1\times L^2)_\perp \to (H^1\times L^2)_\perp} <  \infty
\end{equation}
for every $\lambda\in \rho(\cA_\perp)$ holds. The rest of the proof is devoted to show \eqref{lemF6}. 

Let $\lambda\in\rho(\cA_\perp)$ so that $\Re \lambda>0$, and split $\{\lambda\in \C\, |\, \Re \lambda > 0 \}$ into two disjoint sets: 
\[\begin{array}{l}
S_0 = \{\lambda\in \C \ | \ 0\leq \Re \lambda \leq C_0, \ |\im \lambda|\leq C_1\}\quad\text{ and }\\
S_1 = \{\lambda\in \C \ | \ 0\leq \Re \lambda \leq C_0, \ C_1<|\im \lambda|\}
\cup\{\lambda\in \C \ | \ C_0< \Re \lambda \}.  
\end{array}
\]

First, we consider the set $S_0$. Because, $\lambda\to (\lambda-\cA_\perp)^{-1}$ is continuous and the reversed triangle's inequality,
\[
\left|\ \|(\lambda-\cA_1)^{-1}\|-\|(\mu-\cA_1)^{-1}\|\ \right|\leq \|(\lambda-\cA_1)^{-1}-(\mu-\cA_1)^{-1}\|.
\]
holds for every $\lambda,\ \mu\in \rho(\cA_\perp)$, 
it follows that the map $\lambda\to \|(\lambda-\cA_\perp)^{-1}\|$ is continuous. 
Thus, since $S_0\subset\rho(\cA_\perp)$ is compact, there exists $K>0$ such that $\|(\lambda-\cA_\perp)^{-1}\|\le K$ for every $\lambda\in S_0$.  

Second, we consider the set $S_1$. Let,  $u_\perp\in H^2_\perp$, $v_\perp,f_\perp\in H^1_\perp$, $g_\perp\in L_\perp^2$ and $\alpha,\gamma\in \C$. Assume that $U =(u, v) +\alpha (1, -\nu)\partial_x\brt$ and $F =( f, g) + \gamma (1, -\nu)\partial_x\brt$ are such that 
\[
(\lambda-\cA_\perp)U = F.
\]
From the last equation we get 
\[
(\lambda-\cA_\perp)U_\perp = F_\perp \quad\mbox{and} \quad
    \alpha(\lambda + \nu) = \gamma,
\]
where $U_\perp=(u,v)$ and $F_\perp=(f,g)$. 
Thus, by Lemma~\ref{lem:main_ineq}, 
\[
\nwse{U_\perp}\leq C(C_0,C_1,\nu) \nwse{F_\perp},
\]
for every $\lambda\in S_1$.
Hence, 
\[
    \nwse{U} \leq \nwse{U_\perp} + \frac{\nwse{\gamma(1,-\nu)\partial_x\brt}}{|\lambda+\nu|} \leq  \left(C(C_0,C_1,\nu) + \frac{1}{|\lambda+\nu|}\right)\nwse{F}.
\]
Therefore, $\|(\lambda-\cA_\perp)^{-1}\|$ is bounded for every $\lambda\in S_1$, and  \eqref{lemF6} follows from the latter and the estimates for $S_0$. 
\end{proof}

\section{Nonlinear (orbital) stability} \label{sec:nonlinear_stability}
In this section, we study the stability of solutions $\theta(x,t)$ to the Cauchy problem \eqref{reddyneq}, provided they exist, for initial conditions close to the static N\'eel wall profile $\brt$, and we present the proof of Theorem~\ref{maintheorem}. 

By letting $v=\partial_tu$, problem \eqref{reddyneq} can be written as a the following nonlinear system of equations 
\begin{equation}
\label{eq:FOnlpert}
\left\{
\begin{aligned}
\partial_t W &= F(W), \qquad x\in \R,\ t > 0, \\ W(x,0)&=W_0(x), \qquad x\in \R,  
\end{aligned}\right.
\end{equation}
where $W = (\theta,\varphi)^\top$,   $F(W) = (\varphi,-\nu \varphi - \nabla \cE(\theta))^\top$,
and $W_0 = (u_0,v_0)^\top$.

The nonlinear term in $\nabla \cE(\theta)$ is invariant under space translations (see Lemma 2.6 in \cite{Melc03}). Thus, if $\brt$ denotes the phase of the static N\'eel wall, then 
$\nabla \cE(\brt(\cdot +\delta))=0$ for every $\delta\in \R$. Equation \eqref{eq:FOnlpert} inherits this translation symmetry as
\begin{equation}\label{F-invariance} 
F(\phi(\delta)) = 0, \quad \mbox{for } \quad \phi(\delta) = (\brt(\cdot +\delta),0)^\top \quad \text{ and every } \delta\in \R.
\end{equation} 
Thus, the derivative $F$ with respect to $\delta$ is $DF(\phi(\delta))\phi'(\delta) = 0$. Therefore, zero is an eigenvalue of the $DF(\phi(\delta))$ with eigenfunction $\phi'(\delta)$.

The linearization of \eqref{eq:FOnlpert} around the time-independent solution $\phi(\delta)$ is
\begin{equation}
\label{eq:FOlpert}
\left\{
\begin{aligned}
\partial_t V &= \cA^\delta V \qquad x\in \R,\ t > 0, \\ V(x,0)&=V_0(x) \qquad x\in \R,
\end{aligned}
\right.
\end{equation}
where,
\[
\cA^\delta: = \begin{pmatrix} 0 & \Id \\ -\opl^\delta & -\nu\Id \end{pmatrix}, \qquad \mbox{and}\qquad \opl^\delta u= \left. \frac{d}{d\ep}\nabla\cE\left(\brt(\cdot+\delta)+\ep u\right)\right|_{\ep=0}.
\]
The base spaces of $\cA^\delta$ and $\cL^\delta$ are 
$H^1\times L^2$ and $L^2$, respectively. For $delta=0$, we identify 
\[
\cA^0:= \cA, \qquad \mbox{and}\qquad \opl^0:=\opl.
\]
Translations by $\delta$ in the arguments of $\brt$ 
correspond to the left translation operator $T_l(\delta)$, which is an $L^2$-isometry and a $C_0$-semigroup with generator (see \cite{EN00}). Because $\brt\in H^1$, we have that 
\[
\nld{\partial_x \brt(x+\delta)} = \nld{\partial_x T_l(\delta) \brt(x)} = \nld{ T_l(\delta) \partial_x \brt(x)} = \nld{\partial_x \brt(x)}.
\]
Thus,  $T_l(\delta)$ is also a $H^1$-isometry.
Therefore, all results presented in the previous sections for  $\delta=0$ hold for any $\delta\neq 0$. To keep track of the $\delta$ dependence, we use the notation $\Phi_0(\delta)=(\nu\partial_x\brt(\cdot + \delta), \partial_x\brt(\cdot + \delta))$,  $(H^1\times L^2)_{\perp}^\delta = \{F\in H^1\times L^2 \ |\ \pld{F}{\Phi_0(\delta)}=0\}$,  $\cP^\delta$ for the projection operator into $\{\Phi_0(\delta)\}^\perp$, $\cA_\perp^\delta: (H^1\times L^2)_{\perp}^\delta\to (H^1\times L^2)_{\perp}^\delta$,  etcetera.

In particular, due to Theorem~\ref{thm_semigroup_gen_1},  
there exists a unique solution to \eqref{eq:FOlpert} in $(H^1\times L^2)_{\perp}^\delta$ given by the action of 
a $C_0$-semigroup of quasicontraction generated by $\cA^\delta$. Moreover, because of Theorem~\ref{lemmaeight}, there exist constants $M\geq 1$ and $\Tilde{\omega}>0$ such that
\[
\|e^{t\cA_\perp^{\delta}}V_0\|_{H^1\times L^2}\leq Me^{-\tilde{\omega} t}\|V_0\|_{H^1\times L^2}
\]
for every $V_0\in (H^1\times L^2)_\perp^\delta$. It is important to point out that the spectral and growth bounds \emph{are independent of $\delta$} because they depend only on the $H^1$ and $L^2$ norms of $\brt(\cdot+\delta)$.

The proof of Theorem~\ref{maintheorem} nonlinear stability result follows from an implicit function theorem in Hilbert spaces given by Lattanzio {\em et al.} \cite{LMPS16} based on a similar result for Banach spaces presented by Sattinger \cite{Sa76}. We present the former result here to ease the reading.  

\begin{theorem}\label{LPstability}
Let $X$ be a Hilbert space and $I\subset \R$ be an open neighborhood of $\delta=0$. Assume that $F:\cD\subset X \rightarrow X$ and $\phi:I\subset \R\rightarrow \cD$ satisfies $F(\phi) = 0$. If $\cP^\delta$ is the projector onto $\{\phi'(\delta)\}^\perp_X$ and there exist positive constants $C_0,\delta_0,M,\omega,$ and $\gamma$ such that
\begin{enumerate}

\item[(H1)]\label{H1} for every solution $V = V(t,V_0,\delta)$ to \eqref{eq:FOlpert}, 
\[
\|\cP^\delta V(t,V_0,\delta)\|_X\leq C_0e^{-\omega t}\|\cP^\delta V_0\|_X,
\] 

\item[(H2)]\label{H2} $\phi$ is differentiable at $\delta = 0$ with 
\[
 \|\phi(\delta) -\phi(0)-\phi'(0)\delta\|_{X}\leq C_0|\delta|^{1+\gamma},
\]
for $|\delta|<\delta_0$, and 

\item[(H3)]\label{H3} $F$ is differentiable at $\phi(\delta)$ for every $\delta\in(-\delta_0,\delta_0)$ with 
\begin{equation}
\label{eq:H3}
 \|F(\phi(\delta)+W) -F(\phi(\delta))-DF(\phi(\delta))W\|_{X}\leq C_0\|W\|_{X}^{1+\gamma},
 \end{equation}
for  $\|W\|_{X}\leq M$.
\end{enumerate} 
Then there exists $\ep>0$ such that for any $W_0\in B_\ep(\phi(0)) \subset X$ there exists $\delta\in I$  and a positive constant C for which the solution $W(t;W_0)$ to the nonlinear system \eqref{eq:FOnlpert} satisfies 
\[ 
\|W(t,W_0)-\phi(\delta)\|_{X}\leq C\ \|W_0-\phi(0)\|_{X}\ e^{-\omega t}.
\]
\end{theorem}

\subsection{Proof of Theorem \ref{maintheorem}}
Let $X = H^1\times L^2$, $\cD := H^2\times H^1$, and
\[ 
F(W) = \begin{pmatrix} \varphi\\-\nu \varphi - \nabla \cE(\theta)\end{pmatrix}.
\]
Now, the proof of Theorem~\ref{maintheorem} follows from Theorem~\ref{LPstability} by verifying its assumption. 

Let $\phi(\delta) = (\brt_\delta,0)$, we have that $F(\phi(\delta)) = 0$ for every $\delta\in\R$ by \eqref{F-invariance}. Denote by $V(t,V_0,\delta)$ the solution of \eqref{eq:FOlpert} with initial condition $V_0\in \cD$. Assume $\cP^\delta$ is the projection operator into $\{\Phi_0(\delta)\}^\perp$. Then, condition \emph{(H1)} follows from Theorem~\ref{lemmaeight}.

Because $\brt\in H^2$ is a real-valued smooth function, we have that 
\[
\|\phi(\delta)-\phi(0)-\phi'(0)\delta\|_{H^1\times L^2} = \|\brt_\delta-\brt-\partial_x\brt\delta\|_{H^1}. 
\]
By the Taylor series' remainder integral representation,  for the latter term we estimated \[
    |\brt_\delta-\brt-\partial_x\brt\delta|^2 =  
    \delta^4\left|\int_0^1 (1-t)\, \partial^2_x\brt(x+t\delta)\, dt\right|^2
    \leq 
    \delta^4\int_0^1 (1-t)^2\,\left( \partial^2_x\brt(x+t\delta)\right)^2\, dt,
\]
where the last inequality follows from Jensen's inequality.  Integration in $x$ over $\R$ leads to 
\[
\nld{\brt_\delta-\brt-\partial_x\brt\delta}^2 \leq \delta^4 \int_\R\int_0^1 (1-t)^2\,\left( \partial^2_x\brt(x+t\delta)\right)^2\, dt\,dx.
\]
Now, by the positivity of the last equation's integrand,  we change the integration order, using that $\partial^2_x\brt(x+t\delta) = T_l(t\delta)\partial^2_x\brt$ and that $T_l(t\delta)$ is an $L^2$-isometry, to get
\[
    \nld{\brt_\delta-\brt-\partial_x\brt\delta}^2 
    = \delta^4\nld{ \partial^2_x\brt}^2 \int_0^1(1-t)^2\,dt.
\]
A similar argument for $\partial_x\brt$ gives
\[
    \nld{\partial_x\brt_\delta-\partial_x\brt-\partial^2_x\brt\delta}^2 
    = \delta^4\nld{ \partial^3_x\brt}^2 \int_0^1(1-t)^2\,dt.
\]
Therefore, 
\[
\nwse{\phi(\delta)-\phi(0)-\phi'(0)\delta} \leq \frac{\nhu{\partial_x^2 \brt}}{\sqrt{3}} \, \delta^2,
\]
and \emph{(H2)} follows for $\gamma=1$. 

Let $W=(w_1,w_2) \in H^2\times H^1$, then
\[
F(\phi(\delta)) = \begin{pmatrix}
    0 \\ -\nabla \cE\left(\brt_\delta\right)
\end{pmatrix},\quad  F(\phi(\delta)+W) = \begin{pmatrix}
    w_2 \\ -\nu w_2-\nabla \cE\left(w_1+\brt_\delta\right)
\end{pmatrix},
\]
and 
\[
DF(\phi(\delta))W = \cA^\delta W = \begin{pmatrix}
    w_2 \\ -\nu w_2- \cL^\delta w_1
\end{pmatrix}.
\]
Thus, for \eqref{eq:H3}'s right hand side, we get
\[
\nwse{F(\phi(\delta)+W) -F(\phi(\delta))-DF(\phi(\delta))W} =\nld{\nabla \cE(\brt_\delta +w_1)-\nabla\cE(\brt_\delta)-\opl^\delta w_1},
\]
where, by Proposition \ref{propNeelw} \hyperref[propb]{\rm{(b)}},
\[
\begin{aligned}
    \nabla \cE(\brt_\delta+w_1) &= -\partial_x^2\brt_\delta -\partial_x^2w_1-\sin (\brt_\delta+w_1)\, \cT\cos (\brt_\delta+w_1),\\
    -\nabla \cE(\brt_\delta) &=  \partial_x^2\brt_\delta +\sin \brt_\delta\, \cT(\cos \brt_\delta),\\
    -\opl^\delta w_1 &= \partial_x^2 w_1 - \sin \brt_\delta \cT( \sin \brt_\delta w_1)
    +w_1\cos \brt_\delta \cT( \cos \brt_\delta ),
\end{aligned}
\]
and we use the notation $\cT:=1+(-\Delta)^{1/2}$. Let  $\cK :=\nabla \cE(\brt_\delta +w_1)-\nabla\cE(\brt_\delta)-\opl^\delta w_1$, thus 
\[
\begin{aligned}
     \cK &=  -\sin (\brt_\delta+w_1)\, \cT(\cos (\brt_\delta+w_1)) +\sin \brt_\delta\, \cT (\cos \brt_\delta) + \\
     & \quad -\sin \brt_\delta\cT (\sin \brt_\delta w_1) +w_1\cos \brt_\delta \cT (\cos \brt_\delta).
     \end{aligned}
\]
By adding and subtracting $\sin (\brt_\delta+w_1)\, \cT(\cos (\brt_\delta) - w_1\sin (\brt_\delta))$ to the latter, we get $\cK = A_1 + A_2 + A_3 $ where 
\[
\begin{aligned}
     A_1 &:=  -\sin (\brt_\delta+w_1)\, \cT\left(\cos (\brt_\delta+w_1) -\cos \brt_\delta + w_1\sin \brt_\delta\right),\\
     A_2 &:=  (\sin (\brt_\delta+w_1)-\sin \brt_\delta)\cT (w_1 \sin \brt_\delta),\\
     A_3 &:=  -(\sin (\brt_\delta+w_1)-\sin \brt_\delta-w_1\cos\brt_\delta)\cT (\cos \brt_\delta).
\end{aligned}
\]
By standard calculus arguments, we have that 
\begin{equation}
\label{eq:trig_identities}
\begin{aligned}
     |\sin (\brt_\delta+w_1)-\sin \brt_\delta| &\leq  |w_1|, \\
     |\sin (\brt_\delta+w_1)-\sin \brt_\delta-w_1\cos\brt_\delta| &\leq \tfrac{1}{2}w_1^2, \\
     |\cos (\brt_\delta+w_1)-\cos \brt_\delta+w_1\cos\brt_\delta| &\leq \tfrac{1}{2}w_1^2. 
\end{aligned}
\end{equation}
Because  $H^1(\R)$ is a Banach algebra and $w_1\in H^2(\R)\subset H^1(\R)$, we have that $w_1^2\in H^1$.
Moreover, by the Sobolev embedding theorem, 
\[
\nhu{w_1^2}^2 
    \leq 
    \nli{w_1}^2\left(\nld{w_1}^2 + 4\nld{\partial_x w_1}^2\right)
    \leq 
    4\nhu{w_1}^4.
\]
Thus, by \eqref{eq:trig_identities}, we get 
\begin{equation}
\label{eq:final1}
\nld{A_1}\leq C\nhu{\cos (\brt_\delta+w_1) -\cos \brt_\delta + w_1\sin \brt_\delta}
\leq C\nhu{w_1}^2,
\end{equation}
where we use that $\nld{\cT u} \leq C \nhu{u}$ for every $u\in H^1$ in the first inequality. Analogously, 
\[
\nld{A_2}^2 \leq \nld{|w_1|\cT(w_1 \sin \brt_\delta)}^2
\leq  C^2\nli{w_1}^2\nhu{w_1 \sin \brt_\delta}^2,
\]
\[
\nld{A_3}^2 \leq  \nld{\tfrac{1}{2}w_1^2\cT\cos \brt_\delta}^2
\leq  \tfrac{1}{4}C^2\nli{w_1}^4\nhu{\cos \brt_\delta}^2.
\]
and because $\nli{w_1}\leq \nhu{w_1}$ and  $\nhu{w_1\sin \brt_\delta}^2 \leq \|\sin \brt_\delta\|_{W^{1,\infty}}^2\nhu{w_1}^2$,
we get 
\begin{equation}
\label{eq:final2}
 \nld{A_2}\leq C\|\sin \brt_\delta\|_{W^{1,\infty}}\nhu{w_1}^2, \quad \mbox{and}\quad \nld{A_3}\leq \tfrac{1}{2}C\nhu{\cos \brt_\delta}\nhu{w_1}^2. 
\end{equation}
Finally, by combining \eqref{eq:final1} and \eqref{eq:final2}, we conclude that 
\[
\nld{\cK}\leq \nld{A_1} + \nld{A_2} +\nld{A_3}
\leq \tilde{C}\nhu{w_1}^2 \leq \tilde{C}\nwse{W}^2. 
\]
Therefore,  \emph{(H3)} holds for $\gamma=1$ and the proof is complete.
\qed

\section*{Acknowledgements}

A. Capella and R. G. Plaza thank  Professors Yuri Latushkin and Jaime Angulo Pava for enlightening conversations and useful suggestions during a workshop at the Casa Matem\'atica Oaxaca (BIRS-CMO). The work of A. Capella and R. G. Plaza was partially supported by CONAHCyT, M\'exico, grant CF-2023-G-122.
A. Capella also acknowledges DGAPA and UNAM's support. The work of L. Morales was supported by CONAHCyT, M\'exico, through the Program ``Estancias Postdoctorales por M\'exico 2022''.

\bibliography{references}

\bibliographystyle{amsalpha}

\end{document}